\pdfoutput=1
\documentclass[12pt]{amsart}
\usepackage{amsmath}
\usepackage{amssymb}
\usepackage{amsthm}
\usepackage{enumerate}
\usepackage[usenames,dvipsnames]{pstricks}
\usepackage{epsfig}
\usepackage{dsfont}
\usepackage{color}
\usepackage{pst-grad}
\usepackage{pst-plot} 
\usepackage[a4paper]{geometry}
\usepackage[utf8]{inputenc}
\usepackage{color}

\newcommand{\Op}{\ensuremath{\mathcal{L}}}
\newcommand{\Hyp}{\ensuremath{\mathbb{H}}}
\newcommand{\R}{\ensuremath{\mathbb{R}}}

\newcommand{\ds}{\displaystyle}

  \newcommand{\PP}{\mathbb{P}}   \newcommand{\EE}{\mathbb{E}}   \newcommand{\RR}{\mathbb{R}}  \newcommand{\bbS}{\mathbb{S}}  \newcommand{\bbL}{\mathbb{L}}  \newcommand{\bbX}{\mathbb{X}}   \newcommand{\bbH}{\mathbb{H}} 

    \newcommand{\be}{{\bf{e}}}  \newcommand{\bfX}{{\bf{X}}}

\newcommand{\mcF}{{\mathcal{F}}}  \newcommand{\mcC}{{\mathcal{C}}}   \newcommand{\mcV}{{\mathcal{V}}}   \newcommand{\mcM}{\mathcal{M}}

\newcommand{\ep}{\epsilon}

\newcommand{\ssk}{\smallskip}

\newtheorem{theo}{Theorem}[subsection]
\newtheorem{prop}[theo]{Proposition}
\newtheorem{lemma}[theo]{Lemma}
\newtheorem{defi}[theo]{Definition}
\newtheorem{coro}[theo]{Corollary}
\newtheorem{exe}[theo]{Example}

\newtheorem{remark}[theo]{Remark}

\numberwithin{equation}{section}

\title{Kinetic Brownian motion on Riemannian manifolds}

\author{J. Angst}
\address{IRMAR, 263 Avenue du General Leclerc, 35042 RENNES, France}
\email{jurgen.angst@univ-rennes1.fr}
\urladdr{http://perso.univ-rennes1.fr/jurgen.angst/}

\author{I. Bailleul}
\address{IRMAR, 263 Avenue du General Leclerc, 35042 RENNES, France}
\email{ismael.bailleul@univ-rennes1.fr}
\urladdr{http://perso.univ-rennes1.fr/ismael.bailleul/}

\author{C. Tardif}
\address{LPMA, 4, Place Jussieu	 avenue de France, 75005 PARIS, France}
\email{camille.tardif@upmc.fr}
\urladdr{http://www.proba.jussieu.fr/pageperso/tardif/}

\thanks{The research of the second author was partially supported by the program ANR ``Retour Post-Doctorants", under the contract ANR 11-PDOC-0025. The second author thanks the U.B.O. for their hospitality.
The authors benefit from the support of Lebesgue center, ANR Labex LEBESGUE}

\begin{document}

\begin{abstract}
We consider in this work a one parameter family of hypoelliptic diffusion processes on the unit tangent bundle $T^1 \mathcal M$ of a Riemannian manifold $(\mathcal M,g)$, collectively called kinetic Brownian motions, that are random perturbations of the geodesic flow, with a parameter $\sigma$ quantifying the size of the noise. Projection on $\mcM$ of these processes provides random $C^1$ paths in $\mcM$. We show, both qualitively and quantitatively, that the laws of these $\mcM$-valued paths provide an interpolation between geodesic and Brownian motions. This qualitative description of kinetic Brownian motion as the parameter $\sigma$ varies is complemented by a thourough study of its long time asymptotic behaviour on rotationally invariant manifolds, when $\sigma$ is fixed, as we are able to give a complete description of its Poisson boundary in geometric terms.
\end{abstract}

\maketitle

\tableofcontents
\vfill

\pagebreak

\section{Introduction}
\label{SectionIntro}

\subsection{Motivations and related works}
\label{SubsectionIntro}

We introduce in this work a one parameter family of diffusion processes that model physical phenomena with a finite speed of propagation, collectively called \textbf{kinetic Brownian motion}. The need for such models in applied sciences is real, and ranges from molecular biology, to industrial laser applications, see for instance the works \cite{debb, GrothausStilgenbauer1,GrothausStilgenbauer2} and the references therein. As a first step in this direction, we consider here what may be one of the simplest example of such a process and provide a detailed study of its behaviour in a fairly general geometric setting. In the Euclidean space $\R^d$, kinetic Brownian motion with parameter $\sigma$, is simply described as a $C^1$ random path $(x_t)_{t \geq 0}$ with Brownian velocity on the unit sphere, run at speed $\sigma^2$, so
\begin{equation}
\label{EqKbmRd}
\frac{d x_t}{dt} = \dot x_t,\quad\quad \dot x_t = W_{\sigma^2 t},
\end{equation}
for some Brownian motion $W$ on the unit sphere of $\R^d$.

\begin{figure}[ht]
\label{fig.first}
\begin{center}\includegraphics[scale=.35]{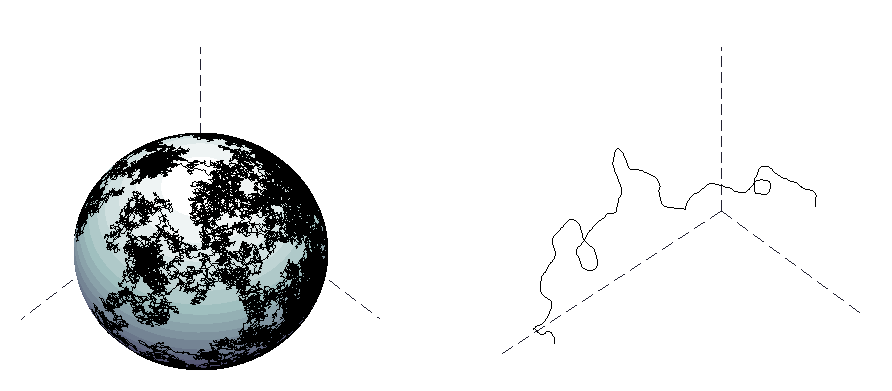}\end{center}
\caption{Kinetic Brownian motion in Euclidean space, velocity on the sphere on the left and the corresponding position on the right.}
\end{figure}

In contrast with Langevin process, whose $\R^d$-valued part can go arbitrarily far in an arbitrarily small amount of time, kinetic Brownian motion provides a bona fide model of random process with finite speed. 
Its definition on a Riemannian manifold follows the intuition provided by its $\R^d$ version, and can be obtained by rolling on $\mcM$ without slipping its Euclidean counterpart. Figure 2 
below illustrates the dynamics of kinetic Brownian motion on the torus, as time goes on.

\begin{figure}[ht]
\label{fig.torus}
\hspace{1cm}\includegraphics[scale=.45]{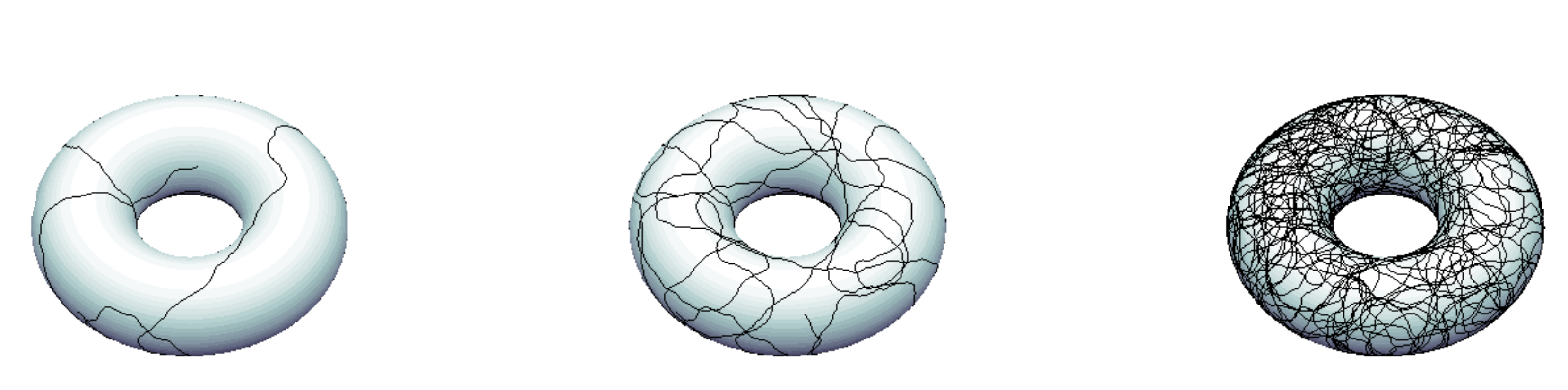}
\caption{Simulations of a the kinetic Brownian motion on the torus over different time intervals.}
\end{figure}

\medskip

We devote most of our efforts in this work in relating the large noise and large time behaviour of the process to the geometry of the manifold. On the one hand, we show that kinetic Brownian motion interpolates between geodesic and Brownian motions, as $\sigma$ ranges from $0$ to $\infty$, leading to a kind of homogenization.  Our use of rough paths theory for proving that fact may be of independent interest. We first prove the interpolation property in the model space $\R^d$, and strengthen the associated convergence results into some rough paths convergence results. The twist here is that once the latter result is proved, the fact that kinetic Brownian motion can be constructed from the rough path lift of kinetic Brownian motion in $\R^d$ by solving a rough differential equation, together with the continuity of the It\^o map, in rough path topology, provides a clean justification of the homogenization phenomenon.

\smallskip

This result strongly echoes Bismut's corresponding result for his hypoelliptic Laplacian \cite{BismutHypoellipticLaplacian} which, roughly speaking, corresponds in its simplest features to replacing the Brownian velocity on the sphere by a velocity process given by an Ornstein-Uhlenbeck process. As a matter of fact, our method for proving the above mentionned homogenization result can also be used to recover the corresponding result for Bismut's hypoelliptic diffusion.

\smallskip

On the other hand, we are able to give a complete description of the Poisson boundary of kinetic Brownian motion when the underlying Riemannian manifold is sufficiently symmetric and $\sigma$ is fixed. This is far from obvious as kinetic Brownian motion is a hypoelliptic diffusion which is non-subelliptic. We take advantage in this task of the powerful \textit{d\'evissage method} that was  introduced recently in \cite{Devissage} as a tool for the analysis of the Poisson boundaries of Markov processes on manifolds. Its typical range of application involves a diffusion $(z_t)_{t\geq 0}$ that admits a subdiffusion $(x_t)_{t\geq  0}$ whose Poisson boundary is known. If the remaining piece $y_t$ of $z_t = (x_t,y_t)$ converges to some random variable $y_\infty$, the d\'evissage method provides conditions that garantee that the invariant sigma field of $(z_t)_{t\geq 0}$ will be generated by $y_\infty$ together with the invariant sigma field of $(x_t)_{t \geq 0}$, see the end of Section \ref{sec.kasymp} where the main results of \cite{Devissage} are recalled. In the present situation, and somewhat like Brownian motion on model spaces, the Poisson boundary of kinetic Brownian motion is described by the asymptotic direction in which the process goes to infinity. It is remarkable, however, that depending on the geometry, kinetic Brownian motion may have a trivial Poisson boundary while Brownian motion will have a non-trivial Poisson boundary.

\smallskip

Kinetic Brownian motion is the Riemannian analogue of a class of diffusion processes on Lorentzian manifolds that was introduced by Franchi and Le Jan in \cite{flj}, as a generalization to a curved setting of a process introduced by Dudley \cite{dudley1} in Minkowski spacetime. These processes model the motion in spacetime of a massive object subject to Brownian fluctuations of its velocity. Despite their formal similarities, the causal structure of spacetime makes relativistic diffusions very different from kinetic Brownian motion, as the reader will find out by reading the litterature on the subject, such as \cite{flj,FranchiLeJanCurvature, FranchiGodel, ismael, BailleulRaugi, BailleulIHP, AngstPoissonRW}, for instance. 

\smallskip

We have organized the article as follows. Kinetic Brownian motion on a given Riemannian manifold $(\mcM,g)$ is introduced formally in Section \ref{SectionDefnKbm} below. Like its elementary $\R^d$-valued version, it lives in the unit tangent bundle of $\mcM$, and it has almost surely an infinite lifetime if the manifold is geodesically complete. Section \ref{SectionInterpolation} is dedicated to proving that the manifold-valued part of a time rescaled version of kinetic Brownian motion converges weakly to Brownian motion as the intensity of the noise, quantified by $\sigma^2$, increases indefinitely, if the manifold is stochastically complete. This is the main content of Theorem \ref{ThmInterpolationGeodBm}, which is proved using rough paths theory. The necessary material on this subject is recalled, so the reader can follow the proof without preliminary knowledge about rough paths. Section \ref{SectionAsymptotics} provides a thourough description of the asymptotic behaviour of kinetic Brownian motion on a rotationnaly invariant manifold, through the identification of its Poisson boundary in a generic setting.

\smallskip

We collect here a number of \textbf{notations} that will be used throughout that work.

\ssk

\begin{itemize}
   \item We shall use Einstein's well-known summation convention and $(\mathcal M,g)$ will denote a $d$-dimensional oriented complete Riemannian manifold, whose unit tangent bundle and orthonormal frame bundle will be denoted respectively by $T^1 \mathcal M$ and $O \mathcal M$. We shall denote by $z = (x,e)$ a generic point of $O \mathcal M$, with $x\in\mathcal M$ and $e : \RR^d\rightarrow T_x\mathcal M$, an orthonormal frame of $T_x\mathcal M$; we write $\pi : O \mathcal M \rightarrow\mathcal M$ for the canonical projection map. Last, we shall denote by $\big(\ep_1,\dots,\ep_d\big)$ the canonical basis of $\RR^d$, with dual basis $\big(\ep_1^*,\dots,\ep_d^*\big)$. 

   \item Denote by $(V_i)_{2\leq i\leq d}$, the canonical vertical vector fields on $O\mathcal M$, associated with the Lie elements $v_i:=\ep_i\otimes\ep_1^*-\ep_1\otimes\ep_i^*$ of the orthonormal group of $\RR^d$. The Levi-Civita connection on $T \mathcal M$ defines a unique horizontal vector field $H_1$ on $O\mathcal M$ such that $(\pi_*H_1)({z}) = e(\ep_1)$, for all $z=(x,e)\in O \mathcal M$. The flow of this vector field is the natural lift to $O \mathcal M$ of the \textit{geodesic flow}. Taking local coordinates $x^i$ on $\mathcal M$ induces canonical coordinates on $O \mathcal M$ by writing 
$$
e_i := e(\ep_i) = \sum_{j=1}^d e_i^j\partial_{x^j}.
$$
Denoting by $\Gamma^k_{ij}$ the Christoffel symbol of the Levi-Civita connection associated with the above coordinates, the vector fields $V_i$ and $H_1$ have the following expressions in these local coordinates 
\[
\begin{array}{ll}
V_{i}(z) &=\ds{e_i^k \frac{\partial}{\partial e_1^k}  - e_1^k \frac{\partial}{\partial e_i^k}}, \quad\quad\quad\quad 2\leq i\leq d,
\\
H_1(z) &= \ds{e_1^i \frac{\partial}{\partial x_i} - \Gamma_{ij}^k(x)  e_1^i e_l^j \frac{\partial}{\partial e_l^k}}.
\end{array}
\]
\end{itemize}

\medskip

\subsection{Definition of kinetic Brownian motion}
\label{SectionDefnKbm}

As said above, kinetic Brownian motion on a $d$-dimensional oriented complete Riemannian manifold $(\mathcal M,g)$, is a diffusion with values in the unit tangent bundle $T^1 \mathcal M$ of $\mathcal M$. In the model setting of $\RR^d$, it takes values in $\RR^d\times\bbS^{d-1}$, and is described as a random $C^1$ path run at unit speed, with Brownian velocity, as described in Equation \eqref{EqKbmRd}. As in the classical construction of Eells-Elworthy-Malliavin of Brownian motion on $\mathcal M$, it will be convenient later to describe the dynamics of kinetic Brownian motion on a general Riemannian manifold $(\mathcal M,g)$ as the projection in $T^1 \mathcal M$ of a dynamics with values in the orthonormal frame bundle $O \mathcal M$ of $\mathcal M$, obtained by rolling without splitting kinetic Brownian motion in $\RR^d$. The following direct dynamical definition in terms of stochastic differential equation provides another description of the dynamics of kinetic Brownian motion which we adopt as a definition. The equivalence of the two point of views is shown in section \ref{SubsubsectionCartan}. In the sequel $\sigma$ will always stand for some non-negative constant which will quantify the strength of the noise in the dynamics of kinetic Brownian motion. 

\begin{defi}\label{def.kbm}
Given $z_0\in O \mathcal M$, the {\bf kinetic Brownian motion} with parameter $\sigma$, started from $z_0$, is the solution to the $O \mathcal M$-valued stochastic differential equation in Stratonovich form
\begin{equation}
\label{EqDefnKbm}
d{z}_t = H_1({z}_t)\,dt + \sigma V_i({z}_t)\,{\circ d}B^i_t,
\end{equation}
started from $z_0$. It is defined a priori up to its explosion time and has generator
$$
\Op_{\sigma} := H_1 + \frac{\sigma^2}{2}  \sum_{j=2}^{d} V_j^2. 
$$
\end{defi}

It is elementary to see that its canonical projection on $T^1 \mathcal M$ is a diffusion on its own, also called \emph{kinetic Brownian motion}. In the coordinate system $(x^i,\dot{x}^i):=(x^i, e_1^i)$ on $T^1\mathcal M$ induced by a local coordinate system on $\mcM$, kinetic Brownian motion satisfies the following stochastic differential equation in It\^o form
\begin{equation}
\label{eqn.system}
\left \lbrace \begin{array}{lll}
d x^i_t &=  \ds{\dot{x}_t^i \, dt},\\
d \dot{x}_t^i &= \ds{ - \Gamma_{jk}^i \, \dot{x}_t^j  \dot{x}_t^k \,dt + \sigma dM_t^i - \frac{\sigma^2}{2}(d-1) \dot{x}^i_t\,dt}.
\end{array}
\right.
\end{equation}
where $1\leq i\leq d$, and where $M_t$ is an $\RR^d$-valued local martingale with bracket
\[
d \langle M^i, M^j \rangle_t = \Big( g^{ij}(x_t) - \dot{x}_t^i   \dot{x}_t^j\Big)dt,
\]
for any $1\leq i,j\leq d$ and $g^{ij}$ stand for the inverse of the matrix of the metric in the coordinates used here.

\smallskip

The readers acquainted with the litterature on relativistic diffusions will recognize in equation \eqref{EqDefnKbm} the direct Riemannian analogue of the stochastic differential equation defining relativistic diffusions in a Lorentzian setting. Despite this formal similarity, the two families of processes have very different behaviours. As a trivial hint that the two situations may differ radically, note that the unit (upper half) sphere $\bbH^d$ in the model space $\R^{1,d}$ of Minkowski spacetime, is unbounded. As a result, there exists deterministic $\bbH^d$-valued paths $\dot x_s$ that explode in a finite time, giving birth to exploding $\R^{1,d}$-valued paths $x_s = x_0 + \int_0^s x_r\,dr$. The work \cite{BailleulExplosion} even gives some reasonnable geometric conditions ensuring the non-stochastic completeness of relativistic diffusions. No such phenomenon can happen in $\R^d$ or on a complete Riemannian manifold for a $\mcC^1$ path defined on a finite open interval, and run at unit speed.

\medskip

\begin{prop}[Non-explosion]
\label{ThrmNonExplosion}
Assume the Riemannian manifold $(\mathcal M,g)$ is complete. Then kinetic Brownian motion has almost surely an infinite lifetime.
\end{prop}

\smallskip

\begin{proof}
Denote by $\tau$ the lifetime of kinetic Brownian motion $z_t = (x_t,e_t)$, and assume, by contradiction, that $\tau$ is finite with positive probability. Since the $C^1$ path $(x_t)_{0\leq t<\tau}$ has unit speed, it would converge as $t$ tends to $\tau$, on the event $\{\tau<\infty\}$, as a consequence of the completeness assumption on $\mathcal M$. The horizontal lift $\big(\overline{z}_t\big)_{0\leq t<\tau}:= (x_t,\overline{e}_t)_{0\leq t<\tau}$ of $(x_t)_{0\leq t<\tau}$ in $O \mathcal M$ would converge as well. Write ${v}_i$ for the Lie element $\ep_i\otimes \ep_1^*-\ep_1\otimes \ep_i^*$ of the orthonormal group $SO(d)$ of $\RR^d$. We have then $\overline{e}_t =: e_t h_t$ where the process $(h_t) \in SO(d)$ satisfies the stochastic differential equation 
$$
dh_t = -{v}_i h_t\,{\circ d}B^i_t,
$$
in particular $(h_t)$ is well defined for all times $t\geq 0$. Consequently, both processes $(e_t)$ and $({z}_t)$ should converge as $t$ tends to $\tau$, contradicting the necessary explosion of ${z}_t$ .
\end{proof}

\medskip

From now on we shall assume that the Riemannian manifold $(\mathcal M,g)$ is \textit{complete}, and turn in the next section to the study of kinetic Brownian motion as a function of $\sigma$.


\section{From geodesics to Brownian paths}
\label{SectionInterpolation}

Let emphasize the dependence of kinetic Brownian motion on the parameter $\sigma$ by denoting it  ${z}_t^{\sigma} = \big(x^\sigma_t,e^\sigma_t\big)\in O\mcM$. We show in this section that the family of laws of $x^\sigma_{\cdot}$ provides a kind of interpolation between geodesic and Brownian motions, as $\sigma$ ranges from zero to infinity, as expressed in Theorem \ref{ThmInterpolationGeodBm} below and illustrated in Fig. \ref{fig.interpol} below when the underlying manifold is the 2-dimensional flat torus.

\begin{figure}[ht]
\hspace{1.1cm}\includegraphics[scale=0.2]{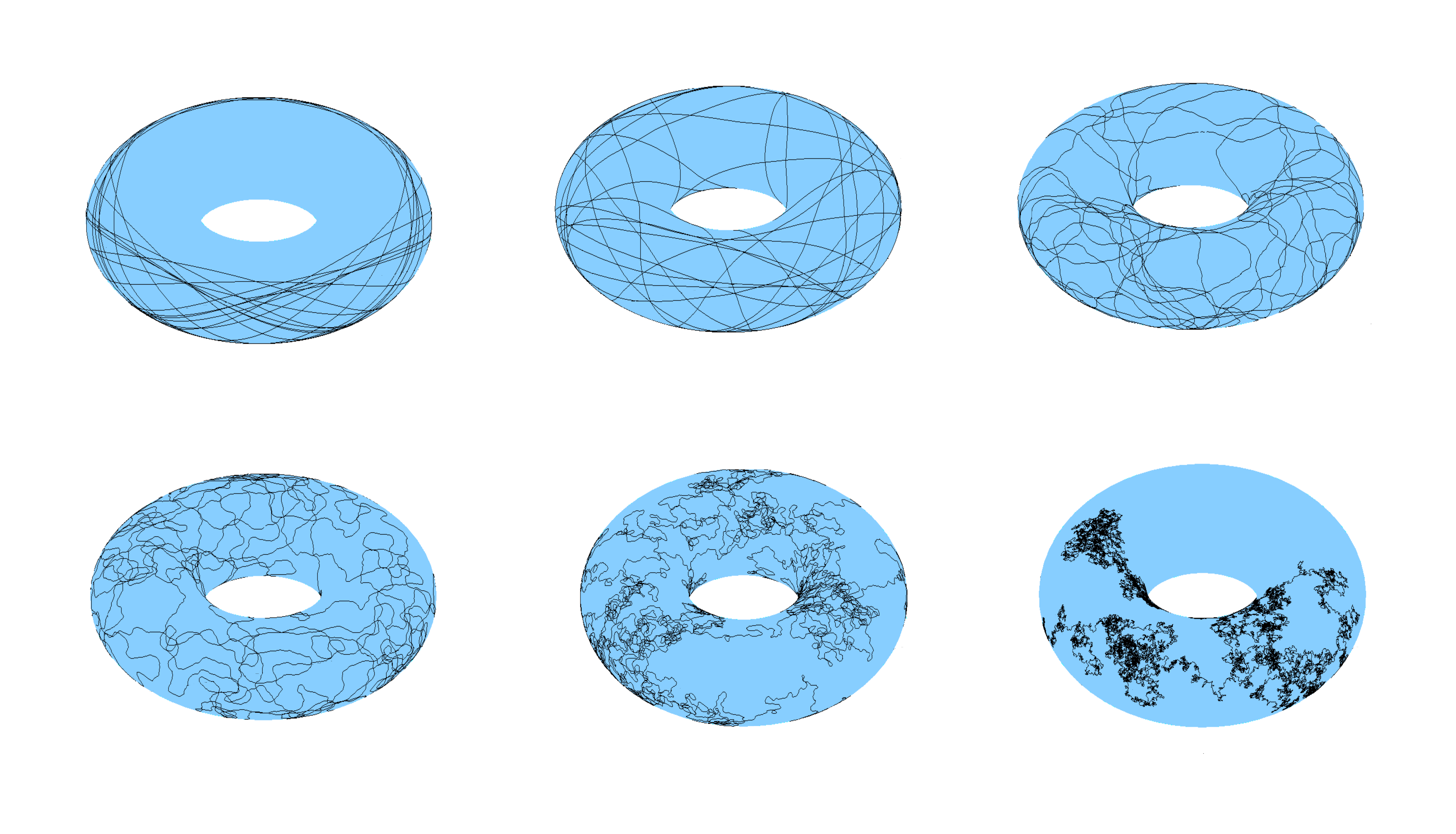}
\label{fig.interpol}
\caption{Kinetic Brownian motion on the flat torus for $\sigma = 10^{-2}, 10^{-1}, 1, 2 ,4, 10$.}
\end{figure}

\subsection{Statement of the results}
\label{SubsectionResultsInterpolation}

To fix the setting, add a cemetary point $\partial$ to $\mathcal M $, and endow the union ${\mathcal M }\sqcup\{\partial\}$ with its usual one-point compactification topology. That being done, denote by $\Omega_0$ the set of continuous paths $\gamma : [0,1]\rightarrow {\mathcal M }\sqcup\{\partial\}$, that start at some reference point $x_0$ and that stay at point $\partial$ if they exit the manifold $\mathcal M$. Let $\mcF:= \bigvee_{t \in [0,1]} \mcF_t$ where $(\mcF_t)_{0\leq t\leq 1}$ stands for the filtration generated by the canonical coordinate process. Denote by $B_R$ the geodesic open ball with center $x_0$  and radius $R$, for any $R>0$. The first exit time from $B_R$ is denoted by $\tau_R$, and used to define a measurable map 
\[
T_R : \Omega_0\rightarrow C\big([0,1], \bar{B}_R\big),
\]
which associates to any path $(\gamma_t)_{0 \leq t \leq 1} \in \Omega_0$ the path which coincides with $\gamma$ on the time interval $\big[0,\tau_R\big]$, and which is constant, equal to $\gamma_{\tau_R}$, on the time interval $\big[\tau_R,1\big]$. The following definition then provides a convenient setting for dealing with sequences of random process whose limit may explode.

\begin{defi}
A sequence $({\PP}_n)_{n\geq 0}$ of probability measures on $\big(\Omega_0,\mcF\big)$ is said to converge locally weakly to some limit probability $\PP$ on $\big(\Omega_0,\mcF\big)$ if the sequence $\PP_n\circ{ T}_R^{-1}$  of probability measures on $C([0,1], \bar{B}_R)$  converges weakly to $\PP\circ{ T}_R^{-1}$, for every $R>0$.
\end{defi}

Equipped with this definition, we can give a precise sense to the above interpolation between geodesic and Brownian motions provided by kinetic Brownian motion.

\begin{theo}[Interpolation]
\label{ThmInterpolationGeodBm}
Assume the Riemannian manifold $(\mcM,g) $ is complete. Given $z_0=\big(x_0,e_0\big)\in O \mathcal M $ we have the two following asymptotics behaviours. 
\begin{itemize}
\item The law of the \emph{rescaled} process $\big(x^\sigma_{\sigma^2t}\big)_{0\leq t\leq 1}$ converges locally weakly under $\PP_{z_0}$ to Brownian motion on $\mathcal M $, run at speed $\frac{4}{d(d-1)}$ over the time interval $[0,1]$, as $\sigma$ goes to infinity. \vspace{0.2cm}

\item The law of the \emph{non-rescaled} process $\big(x^\sigma_t\big)_{0\leq t\leq 1}$ converges locally weakly under $\PP_{z_0}$, as $\sigma$ goes to zero, to a Dirac mass on the geodesic started from $x_0$ in the direction of the first vector of the basis $e_0$.
\end{itemize}
\end{theo}  

Implicit in the above statement concerning the small noise asymptotics is the fact that the geodesic curve is only run over the time interval $[0,1]$; no stochastic completeness assumption is made above. The proof of the second item in the interpolation theorem is trivial, since the generator of the process then converges to the generator of the geodesic flow as it is clear from the Definition \ref{def.kbm} of the process. We shall first prove the result stated in the first item by elementary means in the model Euclidean case, in Section \ref{SubsectionEuclideanInterpolation}. Using the tools of rough paths analysis, we shall see in Section \ref{SubsectionRPInterpolation} that elementary moment estimates allow a strengthening of the weak convergence result of Section \ref{SubsectionEuclideanInterpolation} into the weak convergence of the rough path lift of the Euclidean kinetic Brownian motion to the (Stratonovich) Brownian rough path. To link kinetic Brownian motion in Euclidean space to its Riemannian analogue, we use the fact that the latter can be constructed by rolling on $\mcM$ without slipping the former, using Cartan's development map. This means, from a practical point of view, that one can construct the $\mcM$-valued part of kinetic Brownian motion as the solution of an $\mcM$-valued controlled differential equation equation in which the Euclidean kinetic Brownian motion plays the role of the control. This key fact will enable us to use the continuity of the It\^o-Lyons map associated with Cartan's development map, and transfer in Section \ref{SubsectionGeodesicToBrownian} the weak convergence result proved for the Euclidean kinetic Brownian motion to the curved setting of \emph{any} complete Riemannian manifold. 

\smallskip
Note that the above theorem only involves local weak convergence; it can be strengthened under a very mild and essentially optimal natural assumption.
 
\begin{coro}
\label{cocorico}
If the manifold $(\mcM,g)$ is complete and stochastically complete then the \emph{rescaled} process $\big(x^\sigma_{\sigma^2t}\big)_{0\leq t\leq 1}$ converges in law under $\PP_{z_0}$, as $\sigma$ goes to infinity, to Brownian motion run at speed $\frac{4}{d(d-1)}$ over the time interval $[0,1]$. 
\end{coro} 

\smallskip

X.-M. Li proved recently in \emph{\cite{GeodesicXMLi}} an interpolation theorem similar to Corollary \ref{cocorico}, under stronger geometric assumptions on the base manifold $\mcM$, requiring a positive injetivity radius and a control on the norm of the Hessian of the distance function on some geodesic ball. Her proof rests on a formulation of the weak convergence result in terms of a martingale problem, builds on ideas from homogenization theory, and uses tightness techniques. It is likely that the very robust nature of our proof, based on the rough paths machinery, offers a convenient setting for proving more general homogenization results at a low cost. As a basic illustration, notice for example that the proof below works verbatim with the Levi-Civita connection replaced by any other affine metric preserving connection $\textrm{H} : T\mcM\rightarrow TO\mcM$.  The limit process in $O\mcM$ is not in that case the lift to $O\mcM$ of a Brownian motion on $\mcM$ anymore, but it is still described as the solution to the Stratonovich differential equation
$$
d\be_t = \textrm{H}(\be_t)\,{\circ d}B_t,
$$
where $B$ is an $\R^d$-valued Brownian motion.
\smallskip

Note that a slightly more general family of diffusions on the frame bundle than those given by equation \eqref{EqDefnKbm} is considered in \emph{\cite{GeodesicXMLi}}, where in addition to the Brownian noise in the canonical vertical directions, scaled by a factor $\sigma$, she also considers a vertical constant drift independent of $\sigma$.  It is elementary to adapt our method to this setting.  

\smallskip

\begin{remark}
The idea of using rough paths theory for proving elementary homogenization results as in theorem \ref{ThmInterpolationGeodBm} was first tested in the work \emph{\cite{FrizGassiatLyons}} of Friz, Gassiat and Lyons, in their study of the so-called \emph{physical Brownian motion in a magnetic field.} That random process is described as a $C^1$ path $(x_t)_{0\leq t\leq 1}$ in $\R^d$ modeling the motion of an object of mass $m$, with momentum $p_\cdot = m\dot x_\cdot$, subject to a damping force and a magnetic field. Its momentum satisfies a stochastic differential equation of Ornstein-Uhlenbeck form
$$
dp_t = -\frac{1}{m}\,Mp_tdt + dB_t,
$$
for some matrix $M$ whose eigenvalues all have positive real parts, and $B$ is a $d$-dimensional Brownian motion. While the process $(Mx_t)_{0\leq t \leq 1}$ is easily seen to converge to a Brownian motion, its rough path lift is shown to converge in a rough paths sense in $\bbL^q$, for any $q\geq 2$, to a random rough path \emph{different from} the Brownian rough path. 
\end{remark}

\medskip

The proof of the interpolation theorem \ref{ThmInterpolationGeodBm} and Corollary \ref{cocorico} is split into three steps, performed in the next three subsections. We prove the Euclidean version of Corollary \ref{cocorico} in Section \ref{SubsectionEuclideanInterpolation}, and strengthen that weak convergence result in $\left(C\left([0,1],\R^d\right),\|\cdot\|_\infty\right)$ into a weak convergence result in rough paths topology of the rough paths lift of Euclidean kinetic Brownian motion. This is done in Section \ref{SubsectionRPInterpolation} by using some general compactness criterion on distributions in rough paths space. The point here is that kinetic Brownian motion on any complete Riemannian manifold can be constructed from its Euclidean analogue by solving a controlled differential equation in which the control is the Euclidean kinetic Brownian motion. One can then use Lyons' universal limit theorem, on the continuity of the It\^o-Lyons map, to transfer the weak convergence result of the rough kinetic Brownian motion in $\R^d$ to its Riemannian analogue; this is explained in Section \ref{SubsectionGeodesicToBrownian}, where Corollary \ref{cocorico} is also proved.

\subsection{Proof of the interpolation result in the Euclidean setting}
\label{SubsectionEuclideanInterpolation}

When the underlying manifold $\mathcal M $ is the Euclidean space $\mathbb R^d$, the state space of kinetic Brownian motion becomes $\RR^d\times\bbS^{d-1}$, with coordinates $(x,\dot{x})=(x^i,\dot{x}^i)_{1 \leq i \leq d} \in\RR^d\times\bbS^{d-1}$, inherited from the canonical coordinates on $\R^d\times\R^d$. \if{Its dynamics is simply described by saying that its velocity part $(\dot{x}_t)_{t \geq 0}$ is a Brownian motion on $\bbS^{d-1}$ run at speed $\sigma^2$, while $x_t = x_0 + \int_0^t \dot{x}_s\,ds$.}\fi System \eqref{eqn.system} describing the dynamics of kinetic Brownian motion in a general local coordinate system takes in the present setting the simple form
\begin{equation}\label{eqn.system.eucli}
\left \lbrace \begin{array}{lll}
d x_t^i &=  \ds{\dot{x}_t^i dt},  \\
d \dot{x}^i_t &= \ds{ - \sigma^2 \frac{d-1}{2} \dot{x}^i_tdt + \sigma \sum_{j=1}^d \big(  \delta^{ij} - \dot{x}^i_t  \dot{x}^j_t \big) d W^j_t},
\end{array}
\right.
\end{equation}
for $1\leq i\leq d$, with $W=\big(W^1, \dots, W^d\big)$ a standard $\RR^d$-valued Brownian motion.

\smallskip

\begin{prop}
\label{PropWeakConvergenceBM} 
Given $(x_0,\dot{x}_0) \in\RR^d\times\bbS^{d-1}$, we have the following two asymptotic regimes, in terms of $\sigma$.
\begin{enumerate}
   \item The \emph{non-rescaled} process $(x^\sigma_t)_{0\leq t \leq 1}$ converges weakly to the Dirac mass on the geodesic path $\big(x_0+t \dot{x}_0\big)_{0\leq t\leq 1}$, as $\sigma$ tends to zero.  \vspace{0.1cm}
   \item The \emph{time-rescaled} process $(x^{\sigma}_{\sigma^2 t})_{0\leq t \leq 1}$ converges weakly to a Euclidean Brownian motion with covariance matrix $\frac{4}{d(d-1)}$ times the identity, as $\sigma$ tends to infinity.
\end{enumerate}
\end{prop}

\smallskip

\begin{proof}
The convergence result when $\sigma$ tends to zero is straightforward. We present two proofs of the weak convergence to Brownian motion as $\sigma$ tends to infinity, to highlight the elementary nature of this claim. 
In order to simplify the expressions, let us define for all $t \geq 0$
\[
X^{\sigma}_t:= x_{\sigma^2 t}^{\sigma}.
\]
The first approach takes as a starting point the integrated version of equation \eqref{eqn.system.eucli}, namely
\begin{equation}
\label{EqDynamicsFrakX}
X^{\sigma}_t  = x_0 + \frac{2 }{d-1}\,\frac{1}{\sigma^2} \left( \dot{x}_0-\dot{x}_{\sigma^2 t}^{\sigma} \right) +  M^{\sigma}_{t},
\end{equation}
where $M^{\sigma}_{t}$ is a $d-$dimensional martingale whose bracket is given by
\[
\langle M^{\sigma,i}, M^{\sigma,j}\rangle_t =\frac{4}{(d-1)^2} \,\frac{1}{\sigma^2}   \int_0^{\sigma^2 t}  \big(\delta^{ij} - \dot{x}^{\sigma,i}_s  \dot{x}^{\sigma,j}_s \big)\,ds.
\]
The mid-terms in the right hand side of Equation (\ref{EqDynamicsFrakX}) clearly go to zero when $\sigma$ tends to $\infty$, uniformly in $t$, so that the asymptotic behavior of $X^{\sigma}_t $ is the same as that of the martingale $M^{\sigma}_t$. To analyse that martingale, note that the time-rescaled process $(y_t)_{t \geq 0}:=(\dot{x}_{t/\sigma^2})_{t \geq 0}$ is a standard Brownian motion on $\bbS^{d-1}$, solution of the equation, for each $1\leq i\leq d$
\begin{equation}
\label{EqWidetildeY}
\begin{split}
d y^i_t &= - \frac{d-1}{2} y^i_tdt  + \sum_{j=1}^d \big(  \delta^{ij} - y^i_t y^j_t \big) d B^j_t, \\
                              &=: - \frac{d-1}{2} y^i_tdt  + d N_t^i,
\end{split}
\end{equation}
for some $\RR^d$-valued Brownian motion $B$. The bracket of the martingale $M^{\sigma}$ is simply given in terms of $(y_{t})_{t \geq 0}$ by the formula
\[
\langle M^{\sigma,i}, M^{\sigma,j}\rangle_t =\frac{4}{(d-1)^2} \, \frac{1}{\sigma^4}  \int_0^{\sigma^4 t}  \big(\delta^{ij} - y^i_s  y^j_s \big)\,ds,
\]
so, for a fixed time $t$, the ergodic theorem satisfied by the process $(y_{t})_{t \geq 0}$ entails the almost sure convergence 
\[
\lim_{\sigma \to +\infty}\langle M^{\sigma,i}, M^{\sigma,j}\rangle_t =  \frac{4}{d(d-1)}\,t\,\delta^{ij}.
\]
The result of Proposition \ref{PropWeakConvergenceBM} then follows from the asymptotic version of Knight Theorem, as stated form instance in Theorem 2.3 and Corollary 2.4, pp. 524-525, in the book \cite{revuz} of Revuz and Yor. The second approach consists in starting from the integral representation
\begin{align}
X_{t}^{\sigma} = x_0 +\frac{1}{\sigma^2}  \int_0^{\sigma^4 t} y^i_s ds, \label{xytilde}
\end{align}
where $(y_t)_{t \geq0}$ is the above standard spherical Brownian motion, so that the result can alternatively be seen a consequence of a standard central limit theorem for ergodic diffusions applied to the ergodic process $(y_t)_{t \geq 0}$, see for instance the reference \cite{djalil}.
\end{proof}

We shall prove in Section \ref{SubsectionRPInterpolation} that the weak convergence result of Proposition \ref{PropWeakConvergenceBM} can be enhanced to the weak convergence of the rough path associated with the random $C^1$ path $\left(X^\sigma_{t}\right)_{t \geq 0}$ to the Stratonovich Brownian rough path. This requires the study of the $(\RR^d)^{\otimes 2}$-valued process defined for any $t\geq 0$ by the integral
$$
\mathbb X^{\sigma}_t = \int_0^t X^\sigma_{ s}\otimes d X^\sigma_{s}.
$$

For those readers not acquainted with tensor products, one can simply see the tensor product $a\otimes b$ of two vectors $a,b$ in $\R^d$, as the linear map $x\in\R^d\mapsto \big(\sum_{i=1..d}b^ix^i\big)\,a$.

\begin{prop}
\label{convRP}
The $\RR^d\times(\RR^d)^{\otimes 2}$-valued process $(\bfX_{t}^\sigma )_{t \geq 0}:= \big(X^\sigma_t, \mathbb X^\sigma_t\big)_{t \geq0}$ converges weakly, as $\sigma$ goes to infinity, to the Brownian rough path over $\RR^d$, run at speed $\frac{4}{d(d-1)}$. 
\end{prop}

\medskip

\begin{proof}
It will be convenient to use the representation of $X^{\sigma}_{t}$ given by identity \eqref{xytilde}, for which there is no loss of generality in assuming $x_0=0$. Recall $(N_t)_{t \geq 0}$ stands for the martingale part of $(y_t)_{t \geq 0}$ in Equation \eqref{EqWidetildeY}, seen as an $\R^d$-valued path. We have
\begin{equation*}
\begin{split}
\mathbb X^\sigma_t &= \frac{2}{(d-1)\sigma^4}\,\left(-\int_0^{\sigma^4t}y_s^{\otimes 2}ds  + \frac{2}{d-1}N_{\sigma^4t}^{\otimes 2}\right)\\
&+ \frac{4}{(d-1)^2\sigma^4}\,\left(-N_{\sigma^4t}\otimes y_{\sigma^4t} + \int_0^{\sigma^4t} dN_s\otimes y_s - \int_0^{\sigma^4t} dN_s\otimes N_s \right).
\end{split}
\end{equation*}

As the symmetric part $\frac{1}{2}\left(X^\sigma_{t}\right)^{\otimes 2}$ of $\mathbb X^\sigma_t$ converges to the corresponding symmetric part of the Brownian rough path, by Proposition \ref{PropWeakConvergenceBM}, we are left with proving the corresponding convergence result for the anti-symmetric part of the process $(\mathbb X^\sigma_t)_{t \geq 0}$, namely
\begin{equation}\label{eq.antisym}
\begin{split}
&\frac{4}{(d-1)^2\sigma^4}\,\left( y_{\sigma^4t}\otimes N_{\sigma^4t} - N_{\sigma^4t}\otimes y_{\sigma^4t} + \int_0^{\sigma^4t} dN_s\otimes y_s - y_s\otimes dN_s \right) \\
&+ \frac{4}{(d-1)^2}\,\frac{1}{\sigma^4} \int_0^{\sigma^4t}\big(N_s\otimes dN_s - dN_s\otimes N_s\big).
\end{split}
\end{equation}
As the process $(y_t)_{t \geq 0 }$ lives on the unit sphere, the terms in the first line of expression \eqref{eq.antisym} clearly converge to zero in $\bbL^2$, as $\sigma$ goes to infinity, uniformly on bounded intervals of time. Using the notations introduced in Equation \eqref{EqDynamicsFrakX}, we are thus left with the martingale term
\[
\mathbb M^{\sigma}_t:= \frac{4}{(d-1)^2}\,\frac{1}{\sigma^4} \int_0^{\sigma^4t}\big(N_s\otimes dN_s - dN_s\otimes N_s\big) = \int_0^t \big(M^\sigma_s\otimes dM^\sigma_s - dM^\sigma_s\otimes M^\sigma_s\big).
\]
The weak convergence of this term to the awaited L\'evy area of the Brownian rough path is dealt with by the following lemma, which concludes the proof.
\end{proof}

\begin{lemma}
\label{lem.levybrown}
The $\RR^d\times(\RR^d)^{\otimes 2}$-valued martingale $({\bf M}^{\sigma}_t)_{t \geq 0} := \left(M^{\sigma}_t, \mathbb M^{\sigma}_t \right)_{t \geq 0}$ converges weakly as $\sigma$ goes to infinity, to the process 
\[
\left(B_t, \int_0^t \left (B_s\otimes dB_s - dB_s\otimes B_s\right)\right)_{t \geq 0},
\]
where $B$ is an $\RR^d$-valued Brownian motion run at speed $\frac{4}{d(d-1)}$.
\end{lemma}

\begin{proof}
Given $x\in \R^d$, let $L_x$ stand for the linear map from $\RR^d$ to $\RR^d \times (\R^d)^{\otimes 2}$ defined by the formula
$$
L_{x}v:=\big(v, vx^{*} - xv^{*}\big),
$$
where the star notation is used to denote linear form canonically associated with an element of the Euclidean space $\RR^d$. With this notation, we can write
\[
{\bf M}^{\sigma}_t= \int_0^t L_{M^{\sigma}_s} dM^{\sigma}_s.
\]
Recall that we have already proved that the process $M_{\cdot}^{\sigma}$ converges weakly when $\sigma$ goes to infinity, to a Brownian motion $B_t$ of variance $\frac{4}{d(d-1)}$. It follows from the continuity of the map $x\mapsto L_x$,  that the process $\left(L_{M^{\sigma}_t},M^{\sigma}_t\right)_{0\leq t\leq 1}$ converges weakly to $\left(L_{B_t}, B_t\right)_{0\leq t\leq 1}$. Refering to Theorem 6.2, p. 383 of \cite{JS03}, the result of the lemma will follow from the previous fact if we can prove that the martingales $M^{\sigma}$ are uniformly tight -- see Definition 6.1, p. 277 of \cite{JS03}. But since we are working with continuous martingales whose brackets converge almost surely when $\sigma$ goes to infinity, we can use Proposition 6.13 of \cite{JS03}, p. 379, to obtain the awaited tightness.
\end{proof}

\subsection{From paths to rough paths in the Euclidean setting}
\label{SubsectionRPInterpolation}

Proposition \ref{convRP} shows that the natural lift ${\bfX}^\sigma_{\cdot}$ of $X^\sigma_{\cdot}$ as a weak geometric H\"older $p$-rough path, for $2<p<3$, converges weakly to the Brownian rough path, when seen as an element of a space of continuous paths on $[0,1]$, endowed with the topology of uniform convergence. We show in this section that a stronger convergence result holds, with the rough path metric in place of the uniform topology. This stronger convergence result will be the main ingredient used in Section \ref{SubsectionGeodesicToBrownian} to prove the interpolation theorem  \ref{ThmInterpolationGeodBm}, by relying on Lyons' universal limit theorem.

\smallskip

For the basics of rough paths theory, we refer the reader to the very nice account given in the book \cite{FrizHairer} by Friz and Hairer. Alternative pedagogical accounts of the theory can be found in \cite{M2Course,BaudoinLectureNotes,LejayIntro2}, see also the book \cite{FrizVictoir} of Friz and Victoir for a thourough account of their approach of the theory. Let $2<p<3$ be given.

\begin{prop}
\label{PropRPConvergence}
The weak geometric H\"older $p$-rough path $({\bfX}^\sigma_t)_{0 \leq t \leq 1}$ converges weakly as a rough path to the Brownian rough path.
\end{prop}

The remainder of this section is dedicated to the proof of this statement. Our strategy of proof is simple. Using elementary moment estimates, we show that the family of laws of ${\bfX}^\sigma_{\cdot}$ is tight in some rough paths space. As the rough path topology is stronger than the topology of uniform convergence on bounded intervals, Proposition \ref{convRP} identifies the unique possible limit for these probability measures on the rough paths space, giving the convergence result as a consequence. 

To show tightness, we shall rely on the following Kolmogorov-Lamperti-type compactness criterion; see Corollary A.12 of \cite{FrizVictoir} for a reference. As a shortcut, we will write $z_{ts}$ or $z_{t,s}$ for the increment $z_t-z_s$, for any $s\leq t$, and any path $z_\cdot$ with values in a vector space.

\begin{theo}[Kolmogorov-Lamperti tightness criterion]
\label{ThmKolmoLamperti}
Given any $\frac{1}{3} < \gamma \leq \frac{1}{2}$, consider the laws of $({\bfX}^\sigma_t)_{0\leq t \leq 1} = \big(X^\sigma_t,\mathbb X^\sigma_t\big)_{0 \leq t \leq 1}$, for $\sigma>0$, as probability measures on the metric space $\textrm{\emph{\textsf{RP}}}(\gamma)$ of weak geometric $\gamma$-H\"older rough paths. If the following moment estimates 
\begin{equation}
\begin{split}
&\sup_\sigma\,\EE\Big[\big|X^\sigma_{ts}\big|^q\Big] \leq C_q\,|t-s|^{\frac{q}{2}},  \\
&\sup_\sigma\,\EE\Big[\big|\mathbb X^\sigma_{ts}\big|^q\Big] \leq C_q\,|t-s|^q,
\end{split}
\end{equation}
hold for all $0\leq s\leq t\leq 1$, for some positive constants $C_q$, for all $q\geq 2$, then the family of laws of $({\bfX}^\sigma_t)_{0\leq t \leq 1}$ is tight in $\textrm{\emph{\textsf{RP}}}(\gamma)$.
\end{theo}

We shall prove these two bounds separately, starting with the bound on $X^\sigma_{ts}$. We use for that purpose the representation
$$
X^\sigma_{ts} = \frac{1}{\sigma^2}\int_{\sigma^4s}^{\sigma^4t} y_u\,du,
$$
in terms of a unit speed Brownian motion $(y_t)_{t\geq 0}$ on the sphere, already used in the previous section, together with Equation \eqref{EqWidetildeY} giving the dynamics of $(y_t)_{t\geq 0}$ in terms of an $\RR^d$-valued Brownian motion $(B_t)_{t \geq 0}$. Denoting by $x^i$ the $i^\textrm{th}$ coordinate of a vector $x\in\RR^d$, this gives the equation
$$
(X^\sigma)^i_{ts} = \frac{2}{(d-1)\sigma^2}\left(-y^i_{\sigma^4t,\sigma^4s} + \int_{\sigma^4s}^{\sigma^4t}\big(\delta^{ij} - y^i_uy^j_u\big)\,dB^j_u\right).
$$
So we have
\begin{equation}
\label{EqEstimateFrakX}
\EE\Big[\big|X^\sigma_{ts}\big|^q\Big] \leq \frac{c_q}{\sigma^{2q}} \left\{   \EE\Big[\big|y_{\sigma^4t,\sigma^4s}\big|^q\Big] +  \EE\left[\left|\int_{\sigma^4s}^{\sigma^4t}\big(\delta^{ij} - y^i_uy^j_u\big)\,dB^j_u\right|^q\right]\right\} =: c_q \left\{  (1) + (2)\right\}
\end{equation}
for some positive constant $c_q$ depending only on $q$. Note that term $(2)$ contains an implicit sum over $i$ and $j$. We use Burkholder-Davis-Gundy inequality to deal with it, taking into account the fact that the process $(y_t)_{t \geq 0}$ takes values in the unit sphere. This gives the estimate
\begin{equation}
\label{EqEstimateTerm2}
(2) \leq \sigma^{-2q} \EE\left[\left|\int_{\sigma^4s}^{\sigma^4t}\big(\delta^{ij} - y^i_uy^j_u\big)\,du\right|^\frac{q}{2}\right] \leq 2^\frac{q}{2} |t-s|^\frac{q}{2}.
\end{equation}
For term $(1)$, remark simply that since the identity
$$
y_{ba} = -\frac{d-1}{2}\int_a^by_udu + B_{ba}-\int_a^b y_u\big(y_u^i dB^i_u\big)
$$
holds for all $0\leq a\leq b\leq 1$, for some $\RR^d$-valued Brownian motion $B$, we have
$$
\EE\big[|y_{ba}|^q\big] \leq c_q\left[ \left(\frac{d-1}{2} \right)^q |b-a|^q+|b-a|^\frac{q}{2}+|b-a|^\frac{q}{2}\right],
$$
using the fact that $(y_t)_{t \geq 0}$ lives on the unit sphere, so the process $\int_0^{\cdot} y_u^idW^i_u$ is a real-valued Brownian motion, together with Burkholder-Davis-Gundy inequality. So we have 
$$
\EE\big[|y_{ba}|^q\big] \leq c\,|b-a|^\frac{q}{2}
$$
for all $0\leq a\leq b\leq 1$, for some well-chosen constant $c\geq 2^q$, since $|y_{ba}|\leq 2$. The upper bound $\EE\big[|X_{ts}|^q\big] \leq c_q |t-s|^\frac{q}{2}$, follows then from equations \eqref{EqEstimateFrakX} and \eqref{EqEstimateTerm2}. To deal with the double integral $\mathbb X^\sigma_{ts}$, let us start from Equation \eqref{EqWidetildeY} to write
\begin{equation}
\label{EqEstimateXts}
\begin{split}
\mathbb X_{ts} &= \frac{1}{\sigma^4}\int_{\sigma^4s}^{\sigma^4t} \left(\int_{\sigma^4s}^u y_v dv\right)\otimes y_u\,du \\
                     &= \frac{1}{\sigma^4}\int_{\sigma^4s}^{\sigma^4t} \frac{-2}{d-1}\,y_{u,\sigma^4s}\otimes y_u du + \frac{2}{(d-1)\sigma^4}\int_{\sigma^4s}^{\sigma^4t}  N_{\sigma^4s,u}\otimes y_u du.
\end{split}
\end{equation}
Again, as the process $(y_t)_{t \geq 0}$ lives in the unit sphere, the first term in the right hand side above is bounded above by $\frac{2}{d-1}\,|t-s|$. Let us denote by $(Y_t)_{t \geq 0}$ the $C^1$ path in with values in $\RR^d$ whose velocity is given $(y_t)_{t \geq 0}$. To deal with the term involving the martingale $N$ in equation \eqref{EqEstimateXts}, we use an integration by parts, and the elementary inequality $|a+b|^q\leq c\,\big(|a|^q + |b|^q\big)$, to get the existence of a positive constant $c_q$ depending only on $q$ such that we have
\begin{equation}
\EE\left[\left| \frac{1}{\sigma^4}\int_{\sigma^4s}^{\sigma^4t}  N_{\sigma^4s,u}\otimes y_u du\right|^q\right] 
\leq c_q\big\{(3)+(4)\big\}.
\end{equation}
where we have set
\[
(3):= \EE\left[\left|\frac{N_{\sigma^4t,\sigma^4s}\otimes Y_{\sigma^4t,\sigma^4s}}{\sigma^4}\right|^q\right], \qquad (4):= \EE\left[\left| \frac{1}{\sigma^4}\int_{\sigma^4s}^{\sigma^4t}  dN_u\otimes Y_{u ,\sigma^4s}du\right|^q\right].
\]
There is an absolute positive constant $c$ for which
$$
(3) \leq c\,\EE\left[\left|\frac{N_{\sigma^4t,\sigma^4s}}{\sigma^2}\right|^q\,\left|\frac{Y_{\sigma^4t,\sigma^4s}}{\sigma^2}\right|^q\right] \leq c\,\EE\left[\left|\frac{N_{\sigma^4t,\sigma^4s}}{\sigma^2}\right|^{2q}\right]^\frac{1}{2}  \EE\left[\left|\frac{Y_{\sigma^4t,\sigma^4s}}{\sigma^2}\right|^{2q}\right]^\frac{1}{2}
$$
On the one hand, taking once more into account the fact the the process $(y_t)_{t\geq 0}$ lives on the unit sphere, the Burkholder-Davis-Gundy inequality gives us the upper bound
\begin{equation*}
\EE\left[\left|\frac{N_{\sigma^4t,\sigma^4s}}{\sigma^2}\right|^{2q}\right] = \EE\left[\left|\int_s^t \big(\delta_{\cdot j}-y_{\sigma^4u}y^j_{\sigma^4u}\big)dB^j_u\right|^{2q}\right] \leq c'_q |t-s|^q.
\end{equation*}
On the other hand, we can use the identity
$$
\frac{Y_{\sigma^4t,\sigma^4s}}{\sigma^2} = -\frac{2}{(d-1)\sigma^2}y_{\sigma^4t,\sigma^4s} + \frac{2}{(d-1)\sigma^2}N_{\sigma^4t,\sigma^4s}
$$
and the fact that $\sigma^{-2}N_{\sigma^4t,\sigma^4s}$ has the same law as $\int_s^t\big(\delta^{\cdot j} - y_{\sigma^4u}y^j_{\sigma^4u}\big)dB_u^j$, to see that 
\begin{equation*}
\EE\left[\left|\frac{Y_{\sigma^4t,\sigma^4s}}{\sigma^2}\right|^{2q}\right] \leq c'_q\left\{\EE\left[\left|\frac{y_{\sigma^4t,\sigma^4s}}{\sigma^2}\right|^{2q}\right] + \EE\left[\left|\int_s^t\big(\delta^{\cdot j}-y_{\sigma^4u}y^j_{\sigma^4u}\big)dB_u^j\right|^{2q}\right] \right\}.
\end{equation*}
We recognize in the first term on the right hand side term $(1)$, with $2q$ in the role of $q$, and use the Burkholder-Davis-Gundy inequality to deal with the other expectation. The two results together give an  upper bound of size $ |t-s|^q$, up to a multiplicative constant depending only on $q$, showing that $(3)$ is also bounded above by a constant multiple of $ |t-s|^q$. The inequality of Burkholder-Davis-Gundy is used once more, together with the bound just proved, to estimate term $(4)$ by a constant multiple of
\begin{equation*}
\begin{split}
\EE\left[\left(\int_{\sigma^4s}^{\sigma^4t}\left|\frac{Y_{u,\sigma^4s}}{\sigma^4}\right|^2\,du\right)^\frac{q}{2}\right] &= \EE\left[\left(\int_s^t\left|\frac{Y_{\sigma^4v,\sigma^4s}}{\sigma^2}\right|^2\,dv\right)^\frac{q}{2}\right] \\
&\leq c_q|t-s|^{q-1}\int_s^t \EE\left[\left|\frac{Y_{\sigma^4v,\sigma^4s}}{\sigma^2}\right|^q\right]\,dv \\
&\leq c_q'|t-s|^{\frac{q}{2}-1}\int_s^t |v-s|^\frac{q}{2}\,dv  = c'_q |t-s|^q.
\end{split}
\end{equation*}
Note that all the above estimates hold regardless of the values of $\sigma>0$. All together, these estimates prove that the two conditions of the compactness in Theorem \ref{ThmKolmoLamperti} hold for the family of weak geometric $\gamma$-H\"older rough paths $({\bfX}^\sigma_t)_{0\leq t \leq 1}$. The family of laws of $({\bfX}^\sigma_t)_{0\leq t \leq 1}$ is thus relatively compact on the space $\textrm{\textsf{RP}}(\gamma)$. As its unique possible cluster point is identified by Proposition \ref{convRP}, this proves the weak convergence of the random weak geometric $\gamma$-H\"older rough paths $({\bfX}^\sigma_t)_{0\leq t \leq 1}$ to the Brownian rough path, as stated in Proposition \ref{PropRPConvergence}.

\subsection{From the Euclidean to the Riemannian setting via Cartan's development map}
\label{SubsectionGeodesicToBrownian}

The homogenization result proved in Proposition \ref{PropRPConvergence} puts us in a position to use the machinery of rough differential equations and prove homogenization results for solutions of rough differential equations driven by ${\bfX}^\sigma_{\cdot}$, using to our advantage the continuity of the It\^o map in a rough paths setting. This is in particular the case of kinetic Brownian motion on any complete Riemannian manifold, which can be constructed from kinetic Brownian motion on $\RR^d$, using Cartan's development map. The interpolation theorem \ref{ThmInterpolationGeodBm} will follow from this picture of kinetic Brownian motion as a continuous image of ${\bfX}^\sigma_{\cdot}$. 
Before following that plan, we recall the reader the basics of Cartan's development method and rough differential equations.

\medskip

\subsubsection{Cartan's development map}
\label{SubsubsectionCartan}

As advertized above, one can actually construct kinetic Brownian motion on a complete Riemannian manifold $\mcM$ by rolling on $\mcM$ without slipping its Euclidean analogue. While this will be clear to the specialists from the very definition of kinetic Brownian motion, Cartan's development procedure can be explained to the others as follows. In its classical form, this machinery provides a flexible and convenient way of describing $C^2$ paths on $\mathcal{M}$ from $\R^d$-valued paths; it requires the use of the frame bundle and the horizontal vector fields associated with a choice of connection (Levi-Civita connection presently).

\medskip

 Let H stand for the $TO\mcM$-valued $1$-form on $\R^d$ uniquely characterized by the property 
 $$
 \pi_*\big(\textrm{H}(u)\big)(z)=e(u),
 $$ 
 for any $u\in\R^d$ and $z=(x,e)\in O\mcM$. 
 
\begin{defi} 
Given $z_0=(x_0,e_0)$ in $O\mcM$, \emph{\textbf{Cartan's development of an $\R^d$-valued path}} $(m_t)_{0\leq t\leq 1}$ of class $C^2$ is defined as the solution to the ordinary differential equation
\begin{equation}
\label{EqCartanDvpt}
dz_t = \textrm{\emph{H}}(z_t)\,dm_t
\end{equation}
started from $z_0$. As before, we shall write $z_t=(x_t,e_t)\in O\mcM$. 
\end{defi}

This description of an $O\mcM$-valued path may seem somewhat different from the kind of dynamics described by Equation \eqref{EqDefnKbm} defining kinetic Brownian motion. To make the link clear, assume, without loss of generality, that $(m_t)_{0\leq t \leq 1}$ is run at unit speed, and denote its speed by $(\dot m_t)_{0\leq t \leq 1}$. Recall the definition of the elements ${v}_i$ of the Lie algebra of $SO(d)$ given at the end of Section \ref{SubsectionIntro}. Then, given an orthonormal basis $f_0$ of $\R^d$, with $f_0(\ep_1)=\dot m_0$, solve the $SO(d)$-valued ordinary differential equation 
$$
df_t =\big(f_t(\ep_i),d\dot m_t\big){v}_i(f_t),
$$
started from $f_0$, and define the $\R^{d-1}$-valued control $(h_t)_{0\leq t \leq 1}$, started from zero, by the formula
$$
dh^i_t = \big(f_t(\ep_i),d\dot m_t\big),
$$
for $2\leq i\leq d$. Now, given any orthonormal basis $g_0$ of $T_{m_0}\mcM$ with $g_0(\ep_1) = e_0(\ep_1)$, it is elementary to see that the $O\mcM$-valued path $\big(z'_t\big)_{0\leq t\leq 1}$ obtained by parallel transport of $g_0$ along the path $(x_t)_{0\leq t \leq 1} = (\pi(z_t))_{0\leq t \leq 1}$, satisfies the ordinary differential equation
\begin{equation}
\label{EqCartan1}
dz'_t = H_1\big(z'_t\big)dt + V_i\big(z'_t\big)\,dh^i_t.
\end{equation}
Let emphasize the fact that, by construction 
\[
\pi(z'_t) = x_t,
\] 
Both Equations \eqref{EqCartanDvpt} and \eqref{EqCartan1} make perfect sense with a $C^1$ path $(m_t)_{0\leq t \leq 1}$ whose derivative $(\dot m_t)_{0\leq t \leq 1}$ is a continuous semimartingale; we talk in that case of \textbf{stochastic development}. The reader can consult the book \cite{hsu} of Hsu for a pedagogical account of stochastic differential geometry. Considering in particular the stochastic development in $\mcM$ of kinetic Brownian motion on $\R^d$, provides a control
\begin{equation}
\label{EqCartan2}
dh^i_t = \sigma\,{\circ d}B^i_s,
\end{equation}
for some Stratonovich $(d-1)$-dimensional Brownian increment, in which case Equation \eqref{EqCartan1} is nothing but the equation giving the dynamics of kinetic Brownian motion, identifying $(z^\sigma_t)_{t \geq 0}$ with $(z'_t)_{t \geq 0}$. Since the control given by formula \eqref{EqCartan2} has a law independant of the arbitrary choice of frame $f_0$ satisfying the above constraint, the path $x^{\sigma}_\cdot = \pi(z_\cdot)= \pi(z^\sigma_\cdot) $ is seen to be the projection on $\mcM$ of a kinetic Brownian motion on $\mcM$.

\medskip

Although we know that the $\R^d$-valued part $(m^\sigma_t)_{t \geq 0}$ of kinetic Brownian motion on $\R^d$ converges weakly to some Brownian motion $B$ run at speed $\frac{4}{d(d-1)}$, the setting of It\^o or Stratonovich differential calculus is not robust enough to infer from that result the weak convergence of the developped process $(x^\sigma_t)_{0 \leq t \leq 1}$ to the process $(\pi(w_t) )_{0\leq t \leq 1}$, with $(w_t)_{0\leq t \leq 1}$ solution of the $O\mcM$-valued stochastic differential equation
$$
dw_t = \frac{2}{\sqrt{d(d-1)}}\textrm{\emph{H}}(w_t)\,{\circ d}B_t,
$$
for which $\pi(w_t)$ is nothing else than a scalar time-changed Brownian motion on $\mcM$. This comes from the fundamental lack of continuity of the solution map for stochastic differential equations, in It\^o stochastic integration theory. This missing crucial continuity property is precisely what rough paths theory provides. We give in the next subsection the information on rough paths theory and rough differential equations needed to understand our reasoning, and refer the reader to the litterature on the subject for more insights; see the references below.

\subsubsection{Rough differential equations and the interpolation result}
\label{SubsubsectionRDEs}

Let $\mathcal N$ be a smooth finite dimensional manifold. We adopt in this work the definition of a solution path to an $\mathcal N$-valued rough differential equation given in \cite{BailleulFlows, RDEsBanach}, as it is perfectly suited for our needs. It essentially amounts to requiring from a solution path that it satisfies some uniform Taylor-Euler expansion formulas, in the line of Davie' seminal work \cite{Davie}.  Let $A_1,\dots,{A}_\ell$ be vector fields of class $C^3$ on $\mathcal N$; write $A$ for the vector field-valued 1-form 
$$
A : u\in\R^\ell \mapsto \sum_{i=1}^\ell u^i{A}_i.
$$
Let $2<p<3$ and a weak geometric H\"older $p$-rough path $\bfX=(X,\bbX)$ over $\R^\ell$ be given.

\begin{defi}
An $\mathcal N$-valued continuous path $(m_t)_{0\leq t<\tau}$ is said to solve the rough differential equation
\begin{equation}
\label{EqManifoldValuedRDE}
dm_t = A(m_t)\,{\bfX}(dt)
\end{equation}
if there exists a constant $a>1$ with the following property. For any $0\leq s<\tau$,  there exists an open neighbourhood $\mcV_s$ of $m_s$ such that the Taylor-Euler expansion
$$
f(m_t) = f(m_s) + X^i_{ts}\big({A}_if\big)(m_s) + \bbX^{jk}_{ts}\big({A}_j{A}_kf\big)(m_s) + O\big(|t-s|^a\big) 
$$
holds for all $t$ close enough to $s$, for $m_t$ to belong to $\mcV_s$, and for any function $f$ of class $C^3$ defined on $\mcV_s$.
\end{defi}

The remainder term $O\big(|t-s|^a\big)$ is allowed to depend on $\bfX$ and $f$. It should be clear on this definition that the notions of classical controlled and rough differential equations coincide if the driving rough path is the canonical lift of a $C^1$ path, for the $\bbX_{ts}$ term, of size $(t-s)^2$ in that case, can be put in the remainder. The crucial point for us here will be that if $\bfX$ is the (Stratonovich) Brownian rough path, associated with Brownian motion $B$, then this notion of solution gives back the solution of the Stratonovich differential equation
\[
dm_t = V_i(m_t)\,{\circ d}B^i_t.
\]
See for instance the lecture notes \cite{FrizHairer} of Friz and Hairer, or the book \cite{LyonsQianBook} of Lyons and Qian. Classical results  show that such a rough differential equation has a unique maximal solution started from any given point, and that the \textbf{It\^o map}, that associates to the driving signal $\bfX$ the solution path $x$ to the rough differential equation \eqref{EqManifoldValuedRDE} is continuous in the following sense. Fix $T<\tau$ and cover the compact support of the path $(m_t)_{0\leq t\leq T}$ by finitely many local chart domains $(O'_i)_{1\leq i\leq N}$ and $(O_i)_{1\leq i\leq N}$, with $O'_i\subset O_i$ for all $1\leq i\leq N$. Then, there exists a positive constant $\delta$ such that for every rough path $\bf Y$ $\delta$-close to $\bfX$ in the H\"older rough path distance, the solution path $(y_t)_{0\leq t<\tau}$ to the rough differential equation 
\[
dy_t = {A}(y_t)\,{\bf Y}(dt)
\]
started from $y_0=m_0$, is well-defined on the time interval $[0,T]$ and remains in the open neighbourhood $\bigcup_{i=1}^N O_i$ of the support of $(m_t)_{0\leq t\leq T}$, with $y_t$ in $O_i$ whenever $m_t$ is in $O'_i$. This very strong continuity property implies in particular the following ``transport property'' of weak convergence. We endow the space $\Omega_0$ defined in the introduction of Section \ref{SectionInterpolation} with the topology generated by the elementary open sets 
$$
\Big\{(y_t)_{0 \leq t \leq 1}\in C \big([0,1],{\mathcal N}\sqcup\{\partial\}\big)\,;\,\textrm{dist}(y_t,\gamma_t)<\ep, \textrm{  for all }t\leq \tau_R\Big\},
$$
where $(\gamma_t)_{0 \leq t \leq 1}$ is any element of $\Omega_0$ started from a point of the manifold, $R$ and $\ep$ are arbitrary positive constants, and $\textrm{dist}$ stands for any Riemannian distance function defined in a neighbourhood of the support of $\gamma$. The constant path, equal to $\partial$, is isolated. 

\ssk

\begin{prop}[``Transporting'' weak convergence by the It\^o map]
\label{PropTransportWeakConvergence}
Let $({\bfX}^k)_{k\geq 0}$ be a sequence of random weak geometric $\frac{1}{p}$-H\"older rough paths over $\R^\ell$, whose distribution converges weakly in the rough paths space as $k$ goes to infinity, to the distribution of some random weak geometric H\"older $p$-rough path $\bfX$. Fix $x_0\in\mathcal N$, and consider the solution path $(m^k_t)_{0\leq t \leq 1}$ to the equation
$$
dm^k_t = {A}\big(m^k_t\big)\,{\bfX}^k(dt),
$$
as an element of $\Omega_0$. Then its distribution converges locally weakly to the distribution of the solution path to the rough differential equation
$$
dm_t = {A}\big(m_t\big)\,{\bfX}(dt)
$$
started from $x_0$.
\end{prop}

The interpolation result Theorem \ref{ThmInterpolationGeodBm} follows directly from this fact, together with the construction of kinetic Brownian motion on $\mcM$ as the image of its $\R^d$-valued counterpart by Cartan's development map, as described above, and the results of Section \ref{SubsectionRPInterpolation} on the weak convergence of the rough path lift of kinetic Brownian motion on $\R^d$ to the Stratonovich Brownian rough path. This closes the proof of the interpolation theorem \ref{ThmInterpolationGeodBm}.

\bigskip

The above result on transport of weak convergence only gives locally weak convergence of the image measures. The additional ingredient needed to turn that local weak convergence into genuine weak convergence, as stated in corollary \ref{cocorico}, is a direct consequence of the following elementary lemma, in which we denote by $\mcC_R$ the set of continuous $B_R$-valued paths on $[0,1]$ -- recall $B_R$ stands for a geodesic ball of radius $R$. Denote by $\partial A$ the boundary of a subset $A$ of a topological space.

\begin{lemma}
Let $(\PP_n)_{n\geq 0}$ be a sequence of probability measures on $\Omega_0$ such that 
\begin{itemize}
\item the measures $\PP_n$ are all supported by $C\left([0,1],\mcM\right)$, \vspace{0.2cm}
\item they converge locally weakly to a probability measure $\PP$ on $\Omega_0$, which is also supported on $C\left([0,1],\mathcal{M}\right)$, and such that $\PP\big(\partial \mcC_R\big)=0$, for all $R>0$.
\end{itemize}
Then the sequence $\PP_n$ converge weakly to $\PP$ in $C\left([0,1], \mathcal{M}\right)$.
\end{lemma}

\begin{proof}
First note that given any Borel set $A$ of $C\big([0,1], \mcM\big)$ such that $\PP (\partial A) = 0$, we have
\begin{align*}
\big\vert \PP_n (A) -\PP (A) \big\vert \leq \big\vert \PP_n(A\cap \mcC_R) - \PP (A \cap \mcC_R) \big\vert + \big\vert \PP_n(A \cap \mcC_{R}^c ) \big\vert + \big\vert \PP (A \cap \mcC_{R}^c ) \big\vert. 
\end{align*}
Since $C\big([0,1],\mcM\big) = \bigcup_{R>0} \mcC_R$, and since $\PP$ is supported on $C\big([0,1],\mcM\big)$, given any $\varepsilon>0$, we can find $R$ large enough such that $\big\vert 1- \PP(\mcC_R) \big\vert \leq \varepsilon$; in particular $\big\vert \PP (A \cap\,\mcC_{R}^c ) \big\vert \leq \varepsilon$. We have also
\[
 \big\vert \PP_n(A \cap \mcC_{R}^c ) \big\vert \leq \big\vert 1 - \PP_n(\mcC_R) \big\vert.
\]
Moreover since $\PP(\partial \mcC_R )=0$ we obtain that $\PP(\partial (A\cap \mcC_R) ) =0$ and since $\PP_n$ converges locally weakly to $\PP$ we have 
\[
\PP_n\big( T_{R}^{-1} (A \cap \mcC_R) \big)  \underset{n \to + \infty}{\longrightarrow} \PP \big(T_{R}^{-1}( A \cap \mcC_R) \big).
\]
Since $T_{R}^{-1} (A \cap \mcC_R) =A \cap C_R$ it comes
\[
\vert \PP_n(A\cap \mcC_R) - \PP (A \cap \mcC_R) \vert \underset{n \to +\infty}{\longrightarrow} 0
\]
and for the same reason $\PP_n(\mcC_R)$ converges to $\PP(C_R)$. Thus $\limsup_n \big\vert 1 - \PP_n(\mcC_{R}^c ) \big\vert \leq \varepsilon$ and finally
\[
\limsup_{n \to \infty} \big\vert \PP_n(A) -\PP(A) \big\vert \leq 2 \varepsilon
\]
which ends the proof of the lemma.
\end{proof}

\section{Asymptotics on rotationally invariant manifolds}
\label{SectionAsymptotics}

We aim, in this last section, to try and understand the long time asymptotic behaviour of kinetic Brownian motion when $t$ goes to infinity and $\sigma$ is fixed. As for classical Brownian motion on a general Riemannian manifold, there is no hope to fully determine the asymptotic behaviour of kinetic Brownian motion on arbitrary Riemannian manifolds, as even for Brownian motion it is likely to  depend on the base space in a very sensitive way; see for instance the work \cite{marcanton} of Arnaudon, Thalmaier and Ulsamer for the study of the asymptotic behaviour of Brownian motion on  Cartan-Hadamard manifolds. Of crucial importance in the latter study is the fact that the distance to a fixed point defines a one-dimensional subdiffusion. The difficulties are actually a priori greater here as kinetic Brownian motion does not live on the base manifold, but on its unit tangent bundle, so that it is basically $(2d-1)-$dimensional when the base manifold have dimension $d$, and there is no general reason why it should have lower-dimensional subdiffusions. \par

\medskip

The question of the long time asymptotic behaviour of a manifold-valued diffusion process is of course of different nature depending on whether or not the underlying manifold is compact. If it is compact, the question consists mainly in studying the trend to equilibrium of the process, and in relating the rate of convergence to the geometry of the manifold or to the parameters of the process. In our case, the infinitesimal generator of kinetic Brownian motion is hypoelliptic and $T^1\mathcal M$ is compact, if $\mcM$ is compact, and it is elementary to see that for all times $t$ strictly greater than the diameter of $\mcM$, the density $p_t(\cdot, \cdot)$ of the process is uniformly bounded below by a positive time-dependent constant. This ensures that the law of the process converges exponentially fast to the equilibrium measure, which is the Liouville measure on the unitary tangent bundle here. The question of determining the exact rate of convergence to equilibrium, in terms of the geometry of $\mathcal M$ and the parameter $\sigma$ is difficult and will be the object of another work by the authors. \par

\medskip

If the base manifold is non-compact, the question of the long time asymptotics of a diffusion process amounts to find whether it is recurrent or transient, to exhibit some geometric asymptotic random variables associated to the sample paths, and, in the nicest situations, to determine the Poisson boundary of the process.  In order to get some tractable and significant information on the asymptotic behaviour of kinetic Brownian motion, we will restrict ourselves here on the case where the underlying manifold $\mathcal M$ is rotationally invariant. As in the case of classical Brownian motion, symmetries simplifies greatly the study as they allow to exhibit some lower dimensional subdiffusions of the initial process. The class of rotationally invariant manifolds is nevertheless rich enough to have some good idea of the interplay between geometry, through curvature of the manifold, and the asymptotic behaviour of the process. Our goal here is to obtain some natural conditions on the curvature to ensure transience or recurrence of kinetic Brownian motion; if transient to determine its rate of escape, the existence or not of an asymptotic escape angle, and if possible to characterize the Poisson boundary of the process. We compare these conditions and results to those obtained for Brownian motion on rotationally invariant manifolds, in order to highlight the similarities and differences between the two processes.

\medskip

We recall in Section \ref{SectionRotationnalyInvManifolds} the setting of rotationally invariant manifolds and present our results on the asymptotic behaviour of kinetic Brownian motion in Section \ref{SectionResults}, comparing theorem to their Brownian analogue. As said above, we take advantage of this special symmetric setting to exhibit some subdiffusions that are easier to analyse than the full process. The radial component, together with its derivative, provide a $2$-dimensional subdiffusion whose asymptotic behaviour is thouroughly studied in Section \ref{sec.radial}. As a result, we are able to decide on the recurrent/transient character of kinetic Brownian motion, and in the latter case, to exhibit a whole range of behaviours for its escape rate, depending on the expansion factor in the warped product defining the metric. The angular component is studied in Section \ref{sec.angular}, while we show that the asymptotic angle generates the invariant sigma field of the process under rather general geometric assumptions. Working with a hypoelliptic diffusion makes that task highly non-trivial.

\subsection{Rotationally invariant manifolds}
\label{SectionRotationnalyInvManifolds}

Let us recall that a rotationally invariant Riemannian manifold $(\mathcal M, g)$ is a Riemannian manifold such that $\mathcal M$ or $\mathcal M$ minus a point, admits a global polar coordinates system in which the metric $g$ has a warped product structure. Namely, as illustrated in Figure \ref{fig.rotinv} below, $(\mathcal M, g)$ is a rotationnaly invariant if $\mathcal M$ or $\mathcal M \backslash \{ o \}$, where $o$ is a point in $\mathcal M$,  is diffeomorphic to the product  $(0, +\infty ) \times \mathbb{S}^{d-1}$ endowed with its standard polar coordinates $(r, \theta)$ in which the Riemannian metric $g$ takes the form
\[
g= dr^2 + f(r)^2 d\theta^2, \quad (r, \theta) \in  (0, +\infty ) \, \times \mathbb{S}^{d-1},
\]
where $f$ is a positive smooth function on  $(0, +\infty )$, say at least $C^2$, satisfying the classical assumptions $f(0)=0$ and $f'(0)=1$, see e.g.  \cite{stoker} p. 179-183, and $d\theta$ stands for the metric on $\mathbb{S}^{d-1}$ inherited from the Euclidean metric of $\R^d$. We refer to \cite{greene} for a detailed study of such  manifolds. Since we are interested in non compact manifolds here, we shall suppose in the sequel that $f(r)$ goes to infinity as $r$ increases indefinitely. {\bf We shall assume in the sequel that $d \geq 3$}; the case $d=2$ presents no difficulty but requires a separate treatment, see Remark \ref{rem.d2} in Section \ref{sec.reduction} below, and Appendix A.

\begin{figure}[ht]
\includegraphics[scale=0.4]{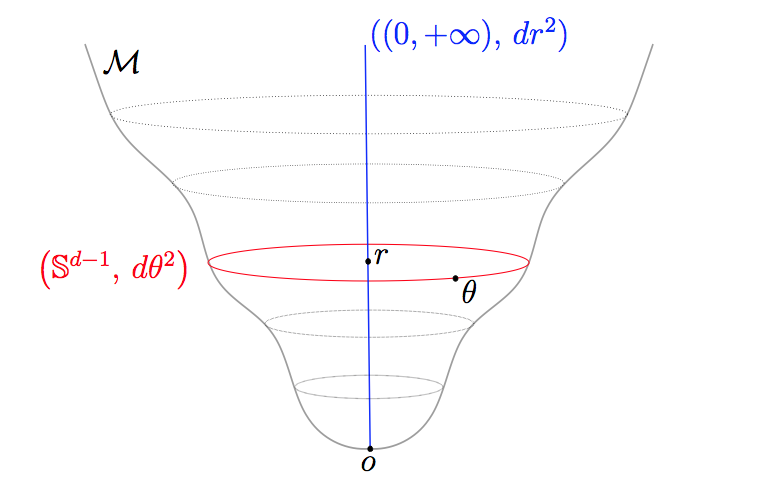}
\caption{Global polar coordinates system on a rotationally invariant manifold.}
\label{fig.rotinv}
\end{figure}

The radial curvature $K$ on a rotationally invariant manifold $\mathcal M$ is given by 
\[
K(r) := -\frac{f''(r)}{f(r)}, 
\]
so that the curvature properties of the manifold are intimately related to the convexity of the warping function $f$. For example, global convexity of $f$ is equivalent to global non positive curvature, and in that case we have $f'(r)/f(r) \geq 1/r$ for all $r > 0$, in particular $f$ is increasing. Indeed, we have 
\[
\left(  \frac{f'}{f}(r)\right)' = \frac{f''}{f}(r) - \left(  \frac{f'}{f}(r)\right)^2, \;\; 
\hbox{so that} \;\; f''(r) \geq 0 \Longrightarrow - \left(  \frac{f'}{f}(r)\right)'  \leq \left(  \frac{f'}{f}(r)\right)^2
\] 
and integrating the ratio, since $f(0)=0$ and $f'(0)=1$, we get $f'(r)/f(r) \geq 1/r$. Moreover, if the warping function $f$ is $\log-$concave, 
the logarithmic derivative $f'/f$ is decreasing and converge to $\ell \in [0, +\infty)$ when $r$ goes to infinity. In that case, for $\varepsilon>0$ and for $r$ sufficiantly large, we have thus $-(\ell+\varepsilon)^2 \leq K(r) \leq 0$. In particular, if $f$ is $\log-$concave and $\ell=0$, the rotationally invariant manifold $(\mathcal M,g)$ is asymptotically flat.

\begin{exe}Here are some classical examples of rotationnaly invariant manifolds and their associated radial curvature.
\begin{enumerate}
\item  The $d$-dimensional Euclidean space minus a point is of course a rotationally invariant manifold. It corresponds to the choice $f(r)= r$ for the warping function and has constant zero curvature. \vspace{0.1cm}

\item The $d$-dimensional hyperbolic space $\Hyp^d$ minus a point, seen as the pseudo-sphere in Minkowski space $\mathbb R^{1,d}$, corresponds to the choice $f(r) = \sinh(r)$. The radial curvature is constant equal to $-1$. \vspace{0.1cm}

\item If $f$ is of polynomial growth, say $f(r)=r^{\beta}$ for $r$ large enough, then the curvature $K(r) = -\frac{\beta(\beta-1)}{r^2}$ has the sign of $1-\beta$ et goes to zero as $r$ goes to infinity, i.e. the manifold is asymptotically flat. \vspace{0.1cm}

\item If $f$ is of subexponential growth, say $f(r)=\exp\left(r^{\beta}\right)$ with $0<\beta<1$ for $r$ large enough, then the curvature 
\[
K(r) = -\frac{\beta(\beta-1)}{r^{2-\beta}} - \frac{\beta^2}{r^{2(1-\beta)}}
\] 
is asymptotically negative and goes to zero as $r$ goes to infinity. \vspace{0.1cm}

\item If the growth of $f$ is more than exponential, e.g. $f(r)=\exp\left(r^{\beta}\right)$ with $\beta>1$ for $r$ large enough, then radial curvature is negative and goes to minus infinity as $r$ goes to infinity.
\end{enumerate}

\end{exe}

\subsection{Statement of the results}
\label{SectionResults}

Before expliciting the long time asymptotic behaviour of kinetic Brownian motion on $\mathcal M$ in terms of the geometry of $\mathcal{M}$, let us recall the corresponding results for the classical Brownian motion $(X_t)_{0 \leq t < \tau}$ on $\mathcal M$.

\subsubsection{Asymptotics of classical Brownian motion}\label{sec.asymp}
On a rotationally invariant manifold, the classical Brownian motion fortunately admits a one-dimensional diffusion, namely the radial subdiffusion $(r_t)_{0 \leq t < \tau}$ when written in polar coordinates $(r_t,\theta_t)_{0 \leq t < \tau}$. Moreover, the angular component $(\theta_t)_{0 \leq t < \tau}$ can be shown to be a time-changed spherical Brownian motion on $\mathbb S^{d-1}$, the clock being a simple additive functional of the radial process, see Example 3.3.3 of \cite{hsu}.  The following classical results are thus simple consequences of the study of this radial one-dimensional diffusion, see for instance \cite{marcanton2} and the references therein. 

\begin{theo}\label{theo.classical}
Let $(r_t,\theta_t)_{0 \leq t < \tau}$ stand for the polar decomposition of the classical Brownian motion with generator $\frac{1}{2} \Delta_{\mathcal M}$ on a rotationally invariant manifold $(\mathcal M, g)$ parametrized by $\left( (0,+\infty)\times \mathbb S^{d-1}, dr^2+f^2(r) d \theta^2\right)$.
\begin{enumerate}
\item The lifetime $\tau$ of the process is almost surely finite or infinite; it is finite almost surely if and only if
\[
\int_1^{+\infty} f^{d-1}(r) \left(\int_{r}^{+\infty} f^{1-d}(\rho)d\rho\right)dr<+\infty.
\]
\item The radial process $r_t$ is transient a.s. if and only if
\[
\int_{1}^{+\infty} f^{1-d}(r)dr <+\infty.
\]
\item The angular process $\theta_t$ converges a.s. to a random point $\theta_{\infty}$ in $\mathbb S^{d-1}$ if and only if
\[
\int_1^{+\infty} f^{d-3}(r) \left( \int_{r}^{+\infty} f^{1-d}(\rho)d\rho\right)dr <+\infty.
\]
\end{enumerate}
\end{theo}

When the process is transient and admits a random escape direction $\theta_{\infty} \in \mathbb S^{d-1}$, it can be shown that this direction is the only non-trivial asymptotic variable, namely the Poisson boundary of the process coincides almost surely with $\sigma(\theta_{\infty})$, see e.g. Theorem 4 of \cite{Devissage} for a simple proof.  Recall that the Poisson boundary of a Markov process  can be defined either as its invariant sigma field or equivalently, via the classical bijection, as the set bounded harmonic functions of its infinitesimal generator.

\begin{theo}\label{theo.poisson}
Let $(r_t,\theta_t)_{t \geq 0}$ stand for the polar decomposition of the classical Brownian motion with generator $\frac{1}{2} \Delta_{\mathcal M}$ on a rotationally invariant manifold $(\mathcal M, g)$ parametrized by $\left( (0,+\infty)\times \mathbb S^{d-1}, dr^2+f^2(r) d \theta^2\right)$ such that
\[
\int_1^{+\infty} f^{d-3}(r) \left( \int_{r}^{+\infty} f^{1-d}(\rho)d\rho\right)dr <+\infty.
\]
Then, the Poisson boundary of the process $(r_t,\theta_t)_{t \geq 0}$ coincides almost surely with $\sigma(\theta_{\infty})$.
\end{theo}

\subsubsection{Asymptotics of kinetic Brownian motion}\label{sec.kasymp}

The following results sum up the long time asymtotics of kinetic Brownian motion on rotationally invariant manifolds. For simplicity and clarity of the statements, we give here the results under non-optimal assumptions. More precise and stronger results are given in Sections \ref{sec.radial}, \ref{sec.angular} and \ref{sec.poisson}, along with the proofs of the corresponding statements. \par

\medskip

As shown in Proposition \ref{ThrmNonExplosion}, the lifetime of kinetic Brownian motion on any complete Riemannian manifold is almost surely infinite. The recurrent and transient character of kinetic Brownian motion are clarified by the following statement.

\begin{theo}\label{theo.radial}
Let $(z_t)_{t \geq 0}=(x_t, \dot{x}_t)_{t \geq 0}$ be the kinetic Brownian motion with values in the unitary tangent bundle $T^1 \mathcal M$ of a rotationally invariant manifold $(\mathcal M, g)$ parametrized by $\left( (0,+\infty)\times \mathbb S^{d-1}, dr^2+f^2(r) d \theta^2 \right)$ and let $(x_t)_{t \geq 0}=(r_t,\theta_t)_{t \geq 0}$ be the polar decomposition of the first projection.
\begin{enumerate}
\item If $K \leq 0$, then the process $r_t$ is transient almost surely.
\item If $K \geq 0$ the process $r_t$ is transient almost surely if and only if 
\[
\int_{1}^{+\infty} f^{1-d}(r)dr <+\infty.
\]
\end{enumerate}
\end{theo}

\begin{remark}
Note that the condition ensuring the transience of the radial component of the kinetic Brownian motion is the same as the one for the classical Brownian motion. Indeed, the strategy of the proof of Theorem \ref{theo.radial}, which is given in Section \ref{sec.radial} below, is to use a change of variables and a time change for the kinetic Brownian motion to make it look ``similar'' to its classical analogue and then use comparison results.
\end{remark}

The next proposition gives a sufficient condition for the convergence of the angular component of kinetic Brownian motion in terms of the warping function: for simplicity, we assume here that it is $\log-$concave. We strongly believe that this condition is necessary but despite our efforts, we did not manage to give a proof of the corresponding statement, see Remark \ref{rem.convangle} below.

\begin{prop}\label{prop.angular}
Let $(z_t)_{t \geq 0}=(x_t, \dot{x}_t)_{t \geq 0}$ stand for kinetic Brownian motion with values in the unitary tangent bundle $T^1 \mathcal M$ of a rotationally invariant manifold $(\mathcal M, g)$, parametrized by $\left( (0,+\infty)\times \mathbb S^{d-1}, dr^2+f^2(r) d \theta^2\right)$. Let $(r_t,\theta_t)_{t \geq 0}$  stand for the polar coordinates of $(x_t)_{t\geq 0}$. Suppose that the warping function $f$ is $\log-$concave and satisfies
\[
\int_1^{+\infty} f^{d-2}(r) \left( \int_{r}^{+\infty} f^{1-d}(\rho)d\rho\right)dr <+\infty.
\]
Then, the angular process $(\theta_t)_{t\geq 0}$ converges almost surely to a random point $\theta_\infty$ on $\mathbb S^{d-1}$.
\end{prop}

\begin{remark}
Note that this time, we obtain a convergence criterion which is different from the one concerning the classical Brownian motion. Namely, the exponent  $d-3$ of the classical case is replaced by $d-2$ here. Therefore, if the integrability condition of Proposition \ref{prop.angular} is indeed necessary, there are situations where the angle $\theta_t$ is convergent for the classical Brownian motion and does not converge for the kinetic Brownian motion. For example, if the warping function is of the form $f(r)=r^{\beta}$ with $\beta>0$, then the angular component converges in the classical case if and only if $\beta>1$, whereas it would converge in the kinetic case if and only if $\beta>2$.
\end{remark}

Let us conclude the synthetic description of the long time asymptotic behaviour of kinetic Brownian motion by expliciting its Poisson boundary, as we did in the Brownian case, in Theorem \ref{theo.poisson} above. 

\begin{theo}\label{theo.kpoisson}
Let $(z_t)_{t \geq 0}=(x_t, \dot{x}_t)_{t \geq 0}$ stand for kinetic Brownian motion with values in the unitary tangent bundle $T^1 \mathcal M$ of a rotationally invariant manifold $(\mathcal M, g)$, parametrized by $\left( (0,+\infty)\times \mathbb S^{d-1}, dr^2+f^2(r) d \theta^2\right)$. Let $(r_t,\theta_t)_{t \geq 0}$  stand for the polar coordinates of $(x_t)_{t\geq 0}$. Suppose that the warping function $f$ is $\log-$concave and satisfies
\[
\int_1^{+\infty} f^{d-2}(r) \left( \int_{r}^{+\infty} f^{1-d}(\rho)d\rho\right)dr <+\infty.
\]
Then the Poisson boundary of the whole process $(z_t)_{t \geq 0}$ coincides almost surely with $\sigma(\theta_{\infty})$.
\end{theo}

The proof of Theorem \ref{theo.kpoisson}, which is given in Section \ref{sec.poisson} below, is based on the \emph{d\'evissage method} introduced by two of the authors in \cite{Devissage}, and which allows to identify Poisson boundaries of diffusion processes with values in manifolds provided the latter admit  enough symmetries so that the original diffusion has some natural subdiffusions. See Section 4 of \cite{Devissage} for various examples of application of the method. 
\if{In our context, it consists in showing that the invariant sigma field of the full kinetic Brownian motion $(x_t, \dot{x}_t)_{t \geq 0} =(r_t, \theta_t, \dot{r}_t, \dot{\theta}_t)_{t \geq 0}$ can be decomposed as the ``sum'' of $\sigma(\theta_{\infty})$ and of the invariant sigma field of the subdiffusion $(r_t, \dot{r}_t)_{t \geq 0}$ which is shown to be trivial by shift coupling methods.}\fi
This method is particularly well suited here since the warped product structure of the manifold gives automatically the existence of lower dimensional subdiffusions; see Section \ref{sec.reduction} below. To help the reader by keeping the present paper as self-contained as possible, let us recall the main result of \cite{Devissage}.

\begin{theo}[{\bf D\'evissage method} -- Theorem A of \cite{Devissage}]\label{theo.devA}
Let $\mathcal N$ be a differentiable manifold and $G$ be a finite dimensional connected Lie group.
Let $(x_t,g_t)_{t \geq 0}$ be a diffusion process with values in $\mathcal N \times G$, starting from $(x,g) \in \mathcal N \times G$. Denote by $\mathbb P_{(x,g)}$ its law. 
Let us suppose that following conditions are satisfied:
\begin{enumerate}
\item D\'evissage condition: the process $(x_t)_{t \geq 0}$ is itself a diffusion process with values in $\mathcal N$. Its own invariant sigma field $\textrm{Inv}((x_t)_{t \geq 0})$ is either trivial or generated by a random variable $\ell_{\infty}$ with values in a separable measure space $(S,\mathcal G, \lambda)$ and the law of $\ell_{\infty}$ is absolutely continuous with respect to $\lambda$.
\item Convergence condition: the second projection $(g_t)_{t \geq 0}$ converges almost surely to a random element $g_{\infty}$ of $G$. 
\item Equivariance condition: the infinitesimal generator $\mathcal L$ of the diffusion $(x_t,g_t)_{t \geq 0}$ is equivariant under the action of $G$ on $C^{\infty}(\mathcal N \times G, \mathbb R)$, i.e. $\forall f \in C^{\infty}(\mathcal N \times G, \mathbb R)$, we have for $(x,g, h) \in \mathcal N \times G \times G$:
\[ 
\mathcal L( f(\cdot,g \cdot ))(x,h) = (\mathcal L f)(x,gh).
\]
\item Regularity condition: all bounded $\mathcal L-$harmonic functions are continuous on the state space $\mathcal N \times G$.
\end{enumerate}
Then, 
$\textrm{Inv}((x_t, g_t)_{t \geq 0})$ and $\textrm{Inv}((x_t)_{t \geq 0}) \vee \sigma(g_{\infty})$ 
coincide up to $\mathbb P_{(x,g)}-$negligeable sets. 
\end{theo}
Note that this result can be generalized to the case where the second projection does not take values in a finite dimensional Lie group $G$ but in a finite dimensional connected co-compact homogeneous space $Y:=G/K$, see Theorem B of \cite{Devissage}.

\if{
\begin{theo}[Theorem  B of \cite{Devissage}]\label{theo.devB}
Let $X$ be a differentiable manifold and $Y$ be a finite dimensional connected co-compact homogeneous space $Y:=G/K$. Let $(x_t,y_t)_{t \geq 0}$ be a diffusion process with values in $X \times Y$ starting from $(x,y) \in X \times Y$ and denote by $\mathbb P_{(x,y)}$ its law. Suppose that there exists a $K$-right equivariant diffusion $(x_t, g_t)_{t \geq 0}$ in $X \times G$ satisfying Hypotheses 1 to 4 of Theorem \ref{theo.devA} above such that under $\mathbb P_{(x,y)}$ the process  $(x_t, y_t)_{t \geq 0}$ has the same law as  $(x_t, \pi(g_t))_{t \geq 0}$ under $\mathbb P_{(x,g)}$ for $g \in \pi^{-1}(\{ y \})$. Then, $\textrm{Inv}((x_t, y_t)_{t \geq 0})$ and $\textrm{Inv}((x_t)_{t \geq 0}) \vee \sigma(y_{\infty})$ 
coincide up to negligeable sets. 
\end{theo}
}\fi

\subsection{Kinetic Brownian motion in polar coodinates}\label{sec.reduction}
The remainder of this article is devoted to proving (finer versions of) Theorem \ref{theo.radial}, Proposition \ref{prop.angular} and Theorem \ref{theo.kpoisson}. We shall exclusively be working, from now on, on a rotationnaly invariant manifold and use polar coordinates. In order to make explicit the long time behaviour of kinetic Brownian motion, we first write down the system of stochastic differential equations satisfied by the polar coordinates of kinetic Brownian motion on a rotationally invariant manifold. We exhibit in particular a lower dimensional subdiffusion that will facilitate its study.
Let us denote by $(r, \theta, \dot{r}, \dot{\theta})$ the polar coordinates in $T^1 \mathcal{M}$. Then, system \eqref{eqn.system} giving the dynamics of kinetic Brownian motion in any local chart reads here
\begin{equation}\label{eqn.system.rot}
\left \lbrace \begin{array}{ll}
dr_t &= \ds{\dot{r}_t dt}, \\
\\
d\theta_t &= \ds{\dot{\theta}_t dt}, \\
\\
d \dot{r}_t &= \ds{\sigma d M^{\dot{r}}_t -\frac{\sigma^2}{2}(d-1) \dot{r}_t dt + \frac{f'}{f}(r_t) \left (1-\dot{r}_t^2 \right ) dt}, \\
\\
d \dot{\theta}^i_t &=\ds{ \sigma d M^{ \dot{\theta}^i}_t -  \frac{\sigma^2}{2}(d-1) \dot{\theta}^i_t dt -2 \frac{f'}{f}(r_t) \dot{r}_t \dot{\theta}^i_{t} dt - \theta^i_t \left(\frac{1-\dot{r}_t^2}{f^2(r_t)}\right)dt}, \quad i=1, \dots, d,
\end{array}\right.
\end{equation}
where $M^{\dot{r}}$ and $M_t^{\dot{\theta}^i}$ are local martingales whose covariance matrix is given by
\begin{equation}
\left \lbrace \begin{array}{ll}
d \ds{\langle M^{\dot{r}}, M^{\dot{r}} \rangle_t }& = \ds{(1-\dot{r}_t^2 )dt}, \\
\\
d \ds{\langle M^{r}, M^{\dot{\theta}^j}  \rangle_t} &= - \ds{\dot{r}_t \dot{\theta}^j_t  dt}, \quad   j=1, \dots, d,\\
\\
d \ds{\langle M^{\dot{\theta}^i}, M^{\dot{\theta}^j}  \rangle_t} &=\ds{\left (\frac{\delta_{i j}-\theta^i_t \theta^j_t}{f^2(r_t)} - \dot{\theta}^i_t \dot{\theta}^j_t \right)dt}, \quad  i,j=1, \dots, d.
\end{array} \right.
\end{equation}
Since the sample paths are parametrized by the arc-length, we have also the following metric relation, relating the radial and angular components:
\begin{equation}\label{eqn.metric}
\dot{r}_t^2 + f(r_t)^2 \big\vert \dot{\theta_t} \big\vert^2 =1.
\end{equation}

\if{
\begin{remark}
The martingales above can be represented by a $d-$dimensional standard Brownian $W=(W^1, \ldots, W^d)$ in the following way: 
\begin{equation}
\left \lbrace \begin{array}{ll}
d M^{\dot{r}}_t & =\ds{ f(r_t) \sum_{i=1}^d \dot{\theta}^i_t d W^i_t},\\
\\
d M^{\dot{\theta}^i}_t&=\ds{\frac{1}{f(r_t)} d W^i_t - \frac{f(r_t)}{1-\dot{r}_t} \dot{\theta}^i_t \dot{\theta}^j_t d W^j_t}, \quad  i=1, \dots, d.\\
\end{array} \right.
\end{equation}
\end{remark}
}\fi

Due to isotropy, kinetic Brownian motion on a rotationally invariant manifold fortunately admits a radial subdiffusion. Contrary to the case of the classical Brownian motion where the radial subdiffusion is one dimensional, the radial subdiffusion of kinetic Brownian motion is $2-$dimensional. 
\begin{prop}
The process $(r_t, \dot{r}_t)_{t \geq 0}$ with values in $(0,+\infty) \times [-1,1]$ is a subdiffusion of the kinetic Brownian motion $(r_t, \dot{r}_t, \theta_t, \dot{\theta}_t)_{t \geq 0} $ in $T^1\mathcal M$.
\end{prop}
\begin{proof}
From Equation (\ref{eqn.system.rot}), the $2-$dimensional radial process is clearly a subdiffusion of the whole process. 
Namely, there exists a real standard Brownian motion $B$ such that $(r_t, \dot{r}_t)$  is solution to the following stochastic differential equations system:
\begin{equation}\label{eqn.radial}\left \lbrace
\begin{array}{ll}
d r_t & = \dot{r}_t dt,\\
\\
d \dot{r}_t &  = \ds{\sigma \sqrt{1-\dot{r}_t^2} dB_t -\frac{\sigma^2}{2}(d-1) \dot{r}_t dt + \frac{f'}{f}(r_t) \left (1-\dot{r}_t^2 \right ) dt}.
\end{array}\right.
\end{equation}
\end{proof}

\if{
\begin{remark}\label{rem.spherique}
Note that, from Equations (\ref{eqn.angular}) and (\ref{eqn.angularcov}), the normalized angular derivative $\dot{\theta}_t/|\dot{\theta}_t|$ is a time changed spherical Brownian motion on $\mathbb S^{d-1}$ parametrized by the clock $\sigma^2 C_t$ where
\begin{equation}
C_t:=  \int_0^t \frac{ds}{1-\dot{r}_s^2}.
\end{equation}
Moreover, this spherical Brownian motion is independant of the radial diffusion as shown at the end of Appendix B.
\end{remark}
}\fi

Let us describe the behaviour of this radial process $(r_t,\dot{r}_t)$ at the boundary points of its state space $(0, +\infty) \times [-1,1]$. 

\begin{prop}\label{pro.lifetime.radial}
If $d \geq 3$, for all starting point $(r_0,\dot{r}_0) \in (0,+\infty) \times [-1,1]$, the stochastic differential equation system (\ref{eqn.radial}) admits a unique strong solution, which is well defined for all $t>0$, and such that $r_t>0$ and $-1<\dot{r}_t<1$ for all $t>0$.
\end{prop}

\begin{proof}
First remark that if $\dot{r}_0^2 < 1$ and $r_0>0$, all the coefficients in (\ref{eqn.radial}) being smooth for $(r,\dot{r}) \in (0,+\infty) \times (-1,1)$, it admits a unique strong solution up to explosion. Moreover, if $\dot{r}_0^2=1$, since $r \mapsto \sqrt{1-\dot{r}^2}$ is $1/2-$H\"older, Equation (\ref{eqn.radial}) also admits a unique strong solution up to explosion and $\dot{r}_t^2<1$ for arbitrary small $t>0$ almost surely. From the metric relation (\ref{eqn.metric}), we have $\dot{r}_t^2 \leq 1$ so that $r_t$ can not go to infinity in finite time almost surely. Otherwise, if $\dot{r}_0^2 < 1$ and $r_0>0$, straightforward It\^o calculus shows that 
\[
f^2(r_t)(1-\dot{r}_t^2) =f^2(r_0)(1-\dot{r}_0^2) \exp \left( -\sigma^2 t  + (d-3)\sigma^2 \int_0^t \frac{\dot{r}_s^2}{1-\dot{r}_s^2} ds +  2 \sigma \int_0^t \frac{\dot{r}_s}{\sqrt{1-\dot{r}_s^2}}dB_s  \right).  
\]
From this, we deduce that, almost surely, the left hand side can not vanish in finite time, in particular $r_t>0$ and $-1<\dot{r}_t<1$ for all $t>0$.
\end{proof}

\begin{remark}\label{rem.d2}
In the case $d=2$, standard comparison arguments show that the endpoints $-1$ and $+1$ can be reached by $r_t$ in finite time almost surely, and that these points are instantaneously reflecting. From relation \eqref{eqn.metric}, the norm $|\dot{\theta }|$ vanishes at such hitting times, so the renormalized process $\dot{\theta}/|\dot{\theta }|$ can not be considered. This is the only reason why the case $d=2$ requires a separate treatment. Nevertheless, the study can be done in the same line as in the case $d \geq 3$, see e.g. Appendix A for a complete treatment of the case where the base manifold is the hyperbolic plane.
\end{remark}

From Equation \eqref{eqn.metric}, Proposition \ref{pro.lifetime.radial} ensures that $|\dot{\theta}_t|>0$ for all $t >0$ almost surely. It will be usefull in the sequel to consider the normalized angular process $(\theta_t, \dot{\theta}_t/|\dot{\theta}_t|)_{t \geq 0}$ with values in $T^1 \mathbb S^{d-1}$. 
Starting from Equation \eqref{eqn.system.rot}, a direct calculation shows that this process satisfies the following system of stochastic differential equations, for $1\leq i \leq d$:

\begin{equation}\label{eqn.angular}
\left \lbrace \begin{array}{ll}
\ds{d \theta_t^i} & = \ds{\frac{\dot{\theta}_t^i}{|\dot{\theta}_t|}  \left(\frac{\sqrt{1-\dot{r}_t^2}}{f(r_t)} \right) dt }, \\
\\
\ds{d \frac{\dot{\theta}_t^i}{|\dot{\theta}_t|} }& = \ds{-\theta^i_t \left(\frac{\sqrt{1-\dot{r}_t^2}}{f(r_t)}\right) dt -  \frac{\sigma^2}{2} (d-2)\frac{\dot{\theta}_t^i}{|\dot{\theta}_t|} \frac{dt}{1-\dot{r}_t^2} + \sigma d N^i_t},
\end{array}
\right.
\end{equation}
where the local martingales are given by 
\[
d N^i_t := \frac{1}{|\dot{\theta}_t|} \left(  d M^{\dot{\theta}^i}_t  - \frac{\dot{\theta}_t^i}{|\dot{\theta}_t|}  \sum_{j=1}^d \frac{\dot{\theta}_t^j}{|\dot{\theta}_t|} d M^{\dot{\theta}^j}_t\right),
\]
so that their covariance matrix reads 
\begin{equation}\label{eqn.angularcov}
d  \left \langle N^i_t, N^j_t \right \rangle= \left( \delta^{ij} - \theta^i_t \theta^j_t -\frac{\dot{\theta}_t^i}{|\dot{\theta}_t|}  \frac{\dot{\theta}_t^j}{|\dot{\theta}_t|} \right) \frac{dt}{1-\dot{r}_t^2}.
\end{equation}
Moreover, for $1 \leq i \leq d$, we have 
\begin{equation}\label{eq.crochet}
d  \left \langle M^{\dot{r}}_t, N^i_t \right \rangle= \frac{1}{|\dot{\theta}_t|} \left(  \left \langle d M^{\dot{r}}, d M^{\dot{\theta}^i}_t \right \rangle  - \frac{\dot{\theta}_t^i}{|\dot{\theta}_t|}  \sum_{j=1}^d \frac{\dot{\theta}_t^j}{|\dot{\theta}_t|} \left \langle d M^{\dot{r}}, d M^{\dot{\theta}^j}_t\right \rangle \right) =0.
\end{equation}
Equivalently, the infinitesimal generator of the full process $(r_t, \dot{r}_t, \theta_t, \dot{\theta}_t/|\dot{\theta}_t|)_{t \geq 0}$ is given by
\begin{equation}
\begin{array}{ll}
\mathcal L & = \ds{\mathcal L_{(r,\dot{r})} + \left( \frac{\dot{\theta}^i}{|\dot{\theta}|} \partial_{\theta^i} - \theta^i \partial_{\dot{\theta}^i /|\dot{\theta}|}\right) \left(\frac{\sqrt{1-\dot{r}^2}}{f(r)}\right) }\\
\\ & +  \ds{\frac{\sigma^2}{2} \left( -(d-2)\frac{\dot{\theta}_t^i}{|\dot{\theta}_t|}  \partial_{\dot{\theta}^i /|\dot{\theta}|} +
 \left( \delta^{ij} - \theta^i \theta^j -\frac{\dot{\theta}^i}{|\dot{\theta}|}  \frac{\dot{\theta}^j}{|\dot{\theta}|} \right)\partial_{\dot{\theta}^i /|\dot{\theta}|} \partial_{\dot{\theta}^j /|\dot{\theta}|} 
\right)   \left(\frac{1}{1-\dot{r}^2}\right)},
\end{array}
\end{equation} 
where $\mathcal L_{(r,\dot{r})}$ is the infinitesimal generator of the subdiffusion $(r_t, \dot{r}_t)_{t \geq 0}$, namely
\[
\mathcal L_{(r,\dot{r})} = \dot{r} \partial_r + \frac{f'}{f}(r) \left (1-\dot{r}^2\right)\partial_{\dot{r}}  +\frac{\sigma^2}{2} \left(  -(d-1) \dot{r} \partial_{\dot{r}} +  \left(1-\dot{r}^2\right)\partial^2_{\dot{r}} \right).
\]
Note that if $\sigma=0$, the above infinitesimal generator is nothing but the generator of the geodesic flow on $T^1 \mathcal M$.

\begin{remark}\label{rem.repmart}
Note that local martingale $N_t=(N^1_t, \ldots, N^d_t)$ can be represented by a standard Euclidean Brownian motion $W_t=(W^1_t, \ldots, W^d_t)$ in the following way:
\[
 d N_t = d W_t - \theta_t \langle \theta_t, d W_t \rangle - \frac{\dot{\theta}_t}{|\dot{\theta}_t|} \left \langle \frac{\dot{\theta}_t}{|\dot{\theta}_t|}, d W_t \right \rangle.
\]
In particular, $d N_t$ is orthogonal to both $\theta_t $ and $\frac{\dot{\theta}_t}{|\dot{\theta}_t|}$.
\end{remark}

\subsection{Asymptotics of the radial components} \label{sec.radial}

We study in this section the recurrence and transience properties of the radial subdiffusion. We give in particular the proof of Theorem \ref{theo.radial} stated in Section \ref{sec.asymp}. 

\subsubsection{Transience/recurrence of the radial process $r_t$}

We first show that the under the single hypothesis that the radial curvature $K(r)$ is globally non-positive, the first projection $r_t$ is almost surely transient, which is the first statement of Theorem \ref{theo.radial}.

\begin{prop}\label{pro.transneg}
Let $(r_t, \dot{r}_t)_{t \geq0}$ be the unique strong solution of Equation (\ref{eqn.radial}) starting from $(r_0, \dot{r}_0)  \in (0,+\infty) \times [-1,1]$. If the warping function $f$ is convex on $(0,+\infty)$ i.e. if the radial curvature $K(r)$ is globally non positive, then $r_t$ goes almost surely to infinity with $t$.
\end{prop}

\begin{remark}\label{rem.curv2}
Let us stress that we are dealing here with \emph{global} convexity and not convexity outside a compact set. Even in the case of the standard Brownian in a rotationally invariant manifold, negative curvature outside a compact set is not sufficient to ensure transience of process, see Remark 2.4 of \cite{marcanton2}.\end{remark}

\begin{proof}
Let $(r_t, \dot{r}_t)$ be the solution of Equation (\ref{eqn.radial}) starting from $(r_0, \dot{r}_0)$.
We have then
\begin{equation}\label{eqn.radial2}
r_t-r_0   =\underbrace{\frac{2 (\dot{r}_0- \dot{r}_t)}{(d-1)\sigma^2} }_{:=I_t}  + \frac{2 }{(d-1)\sigma^2} \underbrace{\int_0^t \frac{f'}{f}(r_s) \left (1-\dot{r}_s^2 \right ) ds}_{:=J_t}+ \frac{2}{(d-1)\sigma} \underbrace{\int_0^t \sqrt{1-\dot{r}_s^2} dB_s}_{:=M_t}.
\end{equation}
Note that the process $I_t$ is bounded since $\dot{r}_t$ is. Otherwise, by  It\^o's formula, we have
\begin{equation}\label{eqn.rpointcarre}
\dot{r}_t^2 = \dot{r}_0^2 - (d-1)\sigma^2 t + d \sigma^2 \underbrace{\int_0^t (1-\dot{r}_s^2) ds}_{=\langle M \rangle_t} + 2  \underbrace{\int_0^t \frac{f'}{f}(r_s) \dot{r}_s \left (1-\dot{r}_s^2 \right ) ds}_{:=K_t} +2 \sigma \underbrace{\int_0^t \dot{r}_s \sqrt{1-\dot{r}_s^2} dB_s}_{:=N_t}.
\end{equation}
Let us first remark that the martingale $M_t$ and the process $J_t$ can not converge simultaneously when $t$ goes to infinity. Indeed, if it was the case, since $\dot{r}_t^2 \leq 1$, the processes $\langle M \rangle_t$, $K_t$ and $N_t$ would also converge and from Equation (\ref{eqn.rpointcarre}), $\dot{r}_t^2$ would tend to minus infinity, hence the contradiction. 
Thus, if we suppose that $M_t$ is convergent, necessarily $J_t$ is divergent and by Equation (\ref{eqn.radial2}), so is $r_t$. 
Let us suppose now that $M_t$ is not convergent so that the following time change $D_t$ and its inverse $D_t^{-1}$ are well defined for all $t \geq 0$ and both go to infinity when $t$ goes to infinity:
\[
D_t:=\sigma^2 \langle M \rangle_t= \sigma^2 \int_0^t (1-\dot{r}_s^2)ds,  \qquad D^{-1}_t:=\inf \{ s>0, D_s >t\}.
\] 
To simplify the expressions, let us define $\rho_t :=r_{D^{-1}_t} $,  $\dot{\rho}_t :=\dot{r}_{D^{-1}_t}$ and introduce the process 
\[
u_t:= \widetilde{\sigma} \rho_t+\dot{\rho}_t, \quad \hbox{where} \quad \widetilde{\sigma}:=\frac{(d-1)\sigma^2 }{2}.
\]
Then, $u_t$ satisfies the equation
\begin{equation}\label{eqn.u}
d u_t= \frac{1}{\sigma^2} \frac{f'}{f}(\rho_t) dt + d B_t, \quad \hbox{i.e.} \quad d u_t= \frac{1}{\sigma^2} \frac{f'}{f}\left( \widetilde{\sigma}^{-1} (u_t-\dot{\rho}_t)\right) dt + d B_t.
\end{equation}
Now consider the process $v_t$ starting from $v_0=u_0$ and solution of the following stochastic differential equation
\[
d v_s = \frac{d-1}{2} \frac{dt}{1+v_t} + d B_t.
\]
Note that $v_t+1$ is then a standard Bessel process of dimension $d \geq 3$, in particular it is almost surely transient. Otherwise, as already noticed at the end of Section \ref{SectionRotationnalyInvManifolds}, thanks to the curvature hypothesis, we have $f'/f(r) \geq 1/r$. Recalling that $\dot{\rho}_t \in [-1,1]$, we deduce that
\begin{align*}
d (v_t-u_t)^+ &= \left[ \frac{d-1}{2} \frac{1}{1+v_t}  - \frac{1}{\sigma^2} \frac{f'}{f}\left( \widetilde{\sigma}^{-1} (u_t-\dot{\rho}_t)\right)\right] \mathds{1}_{v_t>u_t} dt \\
& \leq  \frac{d-1}{2} \left[ \frac{1}{1+v_t}  - \frac{1}{u_t-\dot{\rho}_t} \right] \mathds{1}_{v_t>u_t} dt \\
& = \frac{d-1}{2} \left[ \frac{(u_t-v_t) -(1+\dot{\rho}_t)}{(1+v_t)(u_t-\dot{\rho}_t)} \right] \mathds{1}_{v_t>u_t} dt \leq 0.
\end{align*}
In other words, we have $u_t \geq v_t$ almost surely. We can then conclude that $u_t$ is transient, and since $\dot{\rho}_t$ is bounded, $\rho_t$ is transient and finally the time changed process $r_t = \rho_{D_t}$ is transient. 

\end{proof}

\begin{remark}\label{rem.horloge}
If $r_t$ is transient almost surely and if moreover the warping function $f$ is such that $f'(r)/f(r)$ goes to zero when $r$ goes to infinity, then Equation (\ref{eqn.rpointcarre}) shows that almost surely
\[
\lim_{t \to +\infty}  \frac{1}{t} \int_0^t (1-\dot{r}_s^2) ds = 1- \frac{1}{d}, \quad i.e. \quad  \lim_{t \to +\infty} \frac{1}{t} \int_0^t \dot{r}_s^2 ds =\frac{1}{d}.
\]
\end{remark}

As in the case of Brownian motion on rotationally symmetric manifolds, the transience of $r_t$ can also be caracterized by the integrability of the inverse of the function $f$.
Since the radial subdiffusion is $2-$dimensional here, things are less simple. Nevertheless, we can charaterize transience of kinetic Brownian motion under a simple monotonicity condition on the logarithmic derivative of $f$.

\begin{prop}\label{pro.translog}
Let $(r_t, \dot{r}_t)_{t \geq 0}$ be the solution of Equation (\ref{eqn.radial}) associated with the function $f$ and starting from $(r_0, \dot{r}_0)  \in (0,+\infty) \times [-1,1]$.
\begin{enumerate}
\item Suppose that there exists a smooth function $g$ on $(0,+\infty)$, such that $g$ is positive and $\log-$concave, and such that $f'/f \geq g'/g$. Then, if $\int_1^{+\infty}  g^{1-d} <+\infty$, the radial process $r_t$ is transient almost surely. 
\item On the contrary, suppose that there exists a smooth function $h$ on $(0,+\infty)$, such that $h$ is positive and $\log-$concave, and such that $h'/h \geq f'/f $. Then, if $\int_1^{+\infty}  h^{1-d} =+\infty$, almost surely the radial process $r_t$ is not transient. 
\item In particular, if the warping function $f$ is $\log-$concave, then the radial process $r_t$ is almost surely transient if and only if $\int_1^{+\infty}  f^{1-d} <+\infty$.
\end{enumerate}
\end{prop}

\begin{proof}
Let $(r_t, \dot{r}_t)$ be as in the above statement. With the same notations as in the proof of Proposition \ref{pro.transneg}, the process $u_t$ is then solution of Equation (\ref{eqn.u}), namely
\[
d u_t= \frac{1}{\sigma^2} \frac{f'}{f}\left( \widetilde{\sigma}^{-1} (u_t-\dot{\rho}_t)\right) dt + d B_t.
\]
Let us first suppose that there exists a smooth function $g$ on $(0,+\infty)$, such that $g$ is positive, $\log-$concave, and such that $f'/f \geq g'/g$.
Consider the process $v_t$ starting from $v_0=u_0$ and solution of the stochastic differential equation 
\[
d v_t = \frac{1}{\sigma^2} \frac{g'}{g}\left( \widetilde{\sigma}^{-1} (v_t +1) \right)dt +d B_t.
\]
Classical comparison results then show that $u_t \geq v_t$ almost surely for all $t \geq 0$. To conclude, note that the process 
$\widetilde{v}_t:=\widetilde{\sigma}^{-1} (v_{\widetilde{\sigma}^2 t} +1) $ is solution of the equation 
\[
d \widetilde{v}_{t} = \frac{d-1}{2} \frac{g'}{g}\left( \widetilde{v}_{t}  \right)dt +d B_t,
\]
and $\widetilde{v}_t$ is transient if (and only if) $\int_1^{+\infty}  g^{1-d} <+\infty$, see e.g. Theorem 1.1 p. 208 of \cite{pinsky}, hence the first point.
Suppose now that there exists a smooth function $h$ on $(0,+\infty)$, such that $h$ is positive, $\log-$concave, and $h'/h \geq f'/f$. As above, classical comparison results then show that $u_t \leq v_t$ almost surely for all $t \geq 0$, where $v_t$ is now the solution starting from $v_0=u_0$ of the equation
\[
d v_t = \frac{1}{\sigma^2} \frac{h'}{h}\left( \widetilde{\sigma}^{-1} (v_t -1) \right)dt +d B_t.
\]
Again, to conclude, remark that new rescaled process $\widetilde{v}_t:=\widetilde{\sigma}^{-1} (v_{\widetilde{\sigma}^2 t} -1) $ is now solution of the equation
\[
d \widetilde{v}_{t} = \frac{d-1}{2} \frac{h'}{h}\left( \widetilde{v}_{t}  \right)dt +d B_t,
\]
and $\widetilde{v}_{t}$ is recurrent if (and only if) $\int_1^{+\infty}  h^{1-d} =+\infty$, hence the result. Now if the warping function is $\log-$concave, combining the first two points, we deduce that $r_t$ is transient if and only if $\int_1^{+\infty}  f^{1-d} =+\infty$.
\end{proof}

The last result allows to caracterize the transience of the radial process in non negatively curved rotationally symmetric manifolds. Indeed, recall that $K=-f''/f = -(f'/f)'-(f'/f)^2$, thus if $K\geq 0$, we have necessarily $(f'/f)' \leq 0$ i.e. $f$ is $\log-$concave. From Proposition \ref{pro.translog}, we thus deduce the following results which is the second point in Theorem \ref{theo.radial}.
\begin{coro}
If the radial curvature $K$ is globally non negative, then the radial process $r_t$ is transient if and only if $\int_1^{+\infty}  f^{1-d} <+\infty$.
\end{coro}

\subsubsection{Transience/recurrence of $\dot{r}_t$}

We now describe the asymptotic behaviour of the radial derivative $\dot{r}_t$ in the case where $r_t$ is transient. We first consider the simplest case where the ration $f'/f$ is constant and thus $\dot{r}_t$ is a one dimensional diffusion in $[-1, +1]$. We then have the following lemma:

\begin{lemma}\label{lem.comparexpo}
Let $(r_t, \dot{r}_t)_{t \geq 0}$ be the solution of Equation (\ref{eqn.radial}) associated with the warping function $f$ such that $f'/f(r)\equiv \ell \in \mathbb R$ and starting from $(r_0, \dot{r}_0)  \in (0,+\infty) \times [-1,1]$. Then the process $\dot{r}_t$ is ergodic in the interval $(-1,+1)$ with invariant probability measure
\[
\mu_{\ell}(dx):=\frac{\displaystyle{(1-x^2)^\frac{d-3}{2} e^{\frac{2  \ell}{\sigma^2} x }  }}{\displaystyle{\int_{-1}^{1} (1-x^2)^\frac{d-3}{2} e^{\frac{2  \ell}{\sigma^2} x } dx }} dx.
\]
In particular, if $\ell >0$ we have almost surely
\[
\lim_{t \to +\infty} \frac{r_t}{t} = \int_{-1}^{1} x \,\mu_{\ell}(dx) \in (0,+\infty).
\]
\end{lemma}

\begin{proof}
If $f'/f$ is constant, then Equation (\ref{eqn.radial}) shows that $\dot{r}_t$ is a one-dimensional diffusion and the probability measure $\mu_{\ell}$, whose support is the interval $(-1,+1)$, is invariant, hence the result. The ergodic theorem then ensures that $r_t$ is ballistic almost surely. 
\end{proof}

Now if $f'/f$ is not constant but bounded outside a compact set, we can then deduce that $\dot{r}_t$ is Harris recurrent as soon as $r_t$ is transient.

\begin{prop}\label{pro.recrpoint}
Let $(r_t, \dot{r}_t)_{t \geq 0}$ be the unique strong solution of Equation (\ref{eqn.radial}) starting from $(r_0, \dot{r}_0)  \in (0,+\infty) \times [-1,1]$. Suppose that there exists $R>0$ and $\ell>0$ such that for all $r>R$, we have $| f'(r)/f(r) | < \ell$. Then, if $r_t$ is transient almost surely, the process $\dot{r}_t$ is Harris recurrent in $(-1,+1)$.
\end{prop}

\begin{proof}
If $r_t$ is transient almost surely, using classical one dimensional comparison results, we have $e^{-}_t \leq \dot{r}_t \leq e^+_t$ for $t$ sufficiently large where the two processes $e^{\pm}_t$ are solutions of the equations
\[
d e^{\pm}_t   = \ds{\sigma \sqrt{1-|e^{\pm}_t |^2} dB_t -\frac{\sigma^2}{2}(d-1) e^{\pm}_t  dt \pm  \ell \left (1-|e^{\pm}_t |^2 \right ) dt}.
\]
From Lemma \ref{lem.comparexpo}, both processes are ergodic in $(-1,+1)$, hence the result.
\end{proof}

\begin{coro}
Let $(r_t, \dot{r}_t)_{t \geq 0}$ be the unique strong solution of Equation (\ref{eqn.radial}) starting from $(r_0, \dot{r}_0)  \in (0,+\infty) \times [-1,1]$. Suppose that the warping function $f$ is $\log-$concave with $\lim_{r \to +\infty} f'/f(r) = \ell>0$. Then the process $\dot{r}_t$ is Harris recurrent in $(-1,+1)$. 
\end{coro}

\begin{proof}
Using the monotonicity of $f'/f$, we have $f'/f(r_t) \geq \ell$ almost surely. By classical comparison theorems, we deduce that $\dot{r}_t$ is bounded below by its analogue with $f'/f \equiv \ell$. In particular, by Lemma \ref{lem.comparexpo}, we obtain that $r_t$ is transient almost surely. Then using again comparison results and Lemma \ref{lem.comparexpo}, we deduce that $\dot{r}_t$ is Harris recurrent.
\end{proof}

The almost sure asymptotic behaviour of $\dot{r}_t$ is not clear in the case where $f'/f$ is not bounded above but does not go to infinity. Nevertheless, if the logarithmic derivative is non decreasing outside a compact set and goes to infinity with $r$, then one can show that $\dot{r}_t$ converges to one in probability. A typical example of this last case is a smooth function $f$ such that $f(r)=\exp(r^{\beta})$ for $r>1$ with $\beta>1$ so that $f'/f(r) = \beta r^{\beta-1}$ grows to infinity with $r$.

\begin{prop}
If $f$ is convex and if $f'/f$ is non decreasing outside a compact set and goes to infinity with $r$, then $\dot{r}_t$ converges in probability to one as $t$ goes to infinity.
\end{prop}

\begin{proof}
If $f$ is convex, we know that $r_t$ is transient almost surely by Proposition \ref{pro.transneg}. In particular, since $f'/f(r)$ goes to infinity with $r$, by classical comparison results, we get that for all $\ell>0$ and for all $t$ sufficiently large, $\dot{r}_t\geq u^{\ell}_t$ where $u^{\ell}_t$ is the solution of Equation (\ref{eqn.radial}) associated with a warping function whose logarithmic derivative is constant equal to $\ell$. Therefore, if $\varepsilon>0$, for $t$ sufficiently large, we have 
\[
\mathbb P( \dot{r}_t \geq 1- \varepsilon) \geq \mathbb P\big( u_t^{\ell} \geq 1- \varepsilon\big) \xrightarrow[{t \to +\infty}]{} \mu_{\ell}\big([1-\varepsilon,1]\big),
\]
where the last convergence is a consequence of Lemma \ref{lem.comparexpo} and of the ergodicity of the process $(u_t^{\ell})$.
A direct calculation shows that $\mu_{\ell}([1-\varepsilon,1])$ goes to one as $\ell$ goes to infinity, hence the result.
\end{proof}

\subsubsection{Rate of escape}
We conclude this section dedicated to the study of the radial process by expliciting the rate of escape of $r_t$ when transient.
Under some simple extra assumptions on the behaviour of $f'/f$ at infinity, it is indeed possible to caracterize the speed of divergence of $r_t$. 
With the same notations as in the proof of Proposition \ref{pro.transneg}, let $(r_t, \dot{r}_t)$ be the solution of Equation (\ref{eqn.radial}) associated with the function $f$ and starting from $(r_0, \dot{r}_0)  \in (0,+\infty) \times [-1,1]$ and $u_t$ the associated solution of Equation (\ref{eqn.u}). If we assume that 
$f$ is $\log-$concave, then the same arguments than the ones used in the proof of Proposition \ref{pro.translog} show that the speed of divergence of $\widetilde{\sigma}^{-1} u_{\widetilde{\sigma}^2 t}$ and thus the speed of divergence of $\rho_{\widetilde{\sigma}^2 t}$ is the same as the one of the process $\widetilde{v}_t$ solution of 
\begin{equation}\label{eqn.vtilde}
d \widetilde{v}_{t} = \frac{d-1}{2} \frac{f'}{f}\left( \widetilde{v}_{t}  \right)dt +d B_t.
\end{equation}
Recall that $\rho_{t}$ is obtained from $r_t$ by the simple time change
\[
D_t=\sigma^2 \int_0^t (1-\dot{r}^2_s)ds,
\] 
and from Remark \ref{rem.horloge}, we know that almost surely, the clock $D_t$ is asymptotically linear in $t$ as soon as $f'/f(r)$ goes to zero at infinity. In fine, if $f'/f(r)$ decrease to zero as $r$ goes to infinity, up to a scalar factor, the rate of escape of $r_t$ is the same as the one of $\widetilde{v}_{t}$ solution of Equation 
\eqref{eqn.vtilde} given above. So kinetic Brownian motion happens to escape to $\infty$ at the same speed as Brownian motion, up to a multiplicative constant.

\begin{prop}
Let $(r_t, \dot{r}_t)_{t \geq 0}$ be the unique strong solution of Equation (\ref{eqn.radial}) starting from $(r_0, \dot{r}_0)  \in (0,+\infty) \times [-1,1]$. In the following cases, up to a multiplicative constant, the rate of escape of the process $r_t$ is the same as the one of the solution $\widetilde{v}_{t}$ of Equation 
\eqref{eqn.vtilde} :
\begin{enumerate}
\item If $f$ is of polynomial growth at infinity, namely if $f(r)=r^{\beta}$ for $r$ large enough with $(d-1)\beta>1$, then the process $\widetilde{v}_{t}$ behaves asymptotically as a Bessel process of dimension $d':=1+\beta (d-1)$.
\item If $f$ is of subexponential growth at infinity, e.g. if $f(r)=\exp(r^{\beta})$ for $r$ large enough with $0<\beta<1$, then $\widetilde{v}_{t}$ goes to infinity at speed $t^{\frac{1}{2-\beta}}$, in the sense that the ratio of the two terms converges almost surely to a positive deterministic explicit constant. 
\item If $f$ is of exponential growth at infinity, e.g. $f(r)=e^{c r}$ for $c>0$, then $r_t/t$ converges almost surely to a positive deterministic explicit constant.
\end{enumerate}
\end{prop}

\begin{proof}
In the three cases, we know from Theorem  \ref{theo.classical} that the process $\widetilde{v}_{t}$ is almost surely transient. In the first case where $f(r)=r^{\beta}$, Equation \eqref{eqn.vtilde} reads
\[
d \widetilde{v}_{t} = \frac{(d-1)\beta}{2}  \frac{dt}{\widetilde{v}_{t}} +d B_t,
\]
i.e. $\widetilde{v}_{t}$ is a Bessel process of dimension $d'= (d-1)\beta + 1$, hence the result. In the second case where $f(r)=\exp(r^{\beta})$ for $r$ large enough, we have for $t$ large enough
\begin{equation}\label{eqn.subexpo1}
d \widetilde{v}_{t} = \frac{(d-1)\beta}{2}  \frac{dt}{\widetilde{v}_{t}^{1-\beta}} +d B_t,
\end{equation}
or equivalently if $H(x):=x^{2-\beta}$ : 
\begin{equation}\label{eqn.subexpo2}
d H( \widetilde{v}_{t}) = \frac{2-\beta}{2} \left((d-1)\beta + \frac{1-\beta}{ \widetilde{v}_{t}^{\beta}}\right) dt + (2-\beta)  \widetilde{v}_{t}^{1-\beta} dB_t
\end{equation}
Since $\widetilde{v}_{t}$ is transient almost surely, from Equation \eqref{eqn.subexpo2} we have
\begin{equation}
 H( \widetilde{v}_{t}) =  \left( \frac{\beta(2-\beta)(d-1)}{2}  +o(1) \right) t + M_t, 
\end{equation}
where $M_t$ is a martingale with bracket 
\[
\langle M\rangle_t =  (2-\beta)^2 \int_0^t  \widetilde{v}_{s}^{2- 2\beta} ds.
\]
Otherwise, from Equation \eqref{eqn.subexpo1}, we have that $\widetilde{v}_{t} = o(t)$ almost surely as $t$ goes to infinity.
Injecting this first estimate in the expression for the bracket of $M_t$, we get that $M_t/t$ converges almost surely to zero as soon as $1/2 < \beta <1$, in other words, as $t$ goes to infinity, we have almost surely
\begin{equation}\label{eqn.escape}
\lim_{t\to +\infty}  \frac{\widetilde{v}_{t}^{2-\beta} }{t}=   \frac{\beta(2-\beta)(d-1)}{2}.
\end{equation}
Now if $0<\beta'\leq 1/2 <\beta$, classical comparison results ensure that the solutions $\widetilde{v}_t$ and $\widetilde{v}'_t$ of \eqref{eqn.subexpo1} associated with $\beta$ and $\beta'$ respectively, and starting from a same point, satisfy $\widetilde{v}'_t \leq \widetilde{v}_t$ almost surely for all $t \geq 0$. As before, we have 
\begin{equation}
 H( \widetilde{v}_{t}') =  \left( \frac{\beta'(2-\beta')(d-1)}{2}  +o(1) \right) t + M'_t, 
\end{equation}
where 
\[
\langle M'\rangle_t =  (2-\beta')^2 \int_0^t  \left(\widetilde{v}_{s} '\right)^{2-2\beta'} ds \leq (2-\beta')^2 \int_0^t  \widetilde{v}_{s}^{2-2\beta'} ds.
\]
Since $\beta>1/2$, we know that almost surely $\widetilde{v}_{t} =O(t^{1/(2-\beta)})$ and injecting this new estimate in the expression of $\langle M'\rangle_t$, we get that $M'_t/t$ goes almost surely to zero as $t$ goes to infinity as soon as $1/4<\beta' \leq 1/2$. Iterating the above argument, we eventually obtain that
the almost sure convergence \eqref{eqn.escape} holds true for all $0<\beta<1$. Finally, in the third case where the warping function has exponential growth, the logarithmic  derivative in Equation \eqref{eqn.vtilde} is constant, and $\widetilde{v}_{t}$ is then trivially ballistic, hence the result.
\end{proof}

\subsection{Asymptotics of the angular components} \label{sec.angular}
Let us now describe the asymptotic behavior of the angular components in function of the geometry of the underlying manifold. 
From Equation \eqref{eqn.metric}, we have
\[
\theta_t = \theta_0 + \int_{0}^t \frac{\dot{\theta}_s}{|\dot{\theta}_s|} \frac{\sqrt{1-\dot{r}_s^2}}{f(r_s)}ds.
\]
Thus, if we introduce the new clock $\displaystyle{C_t:=\int_0^t  \frac{\sqrt{1-\dot{r}_s^2}}{f(r_s)}ds}$ with inverse $C_t^{-1}$, we have
\[
\theta_{t} = \theta_0 + \int_{0}^{C_t} \frac{\dot{\theta}_{C_s^{-1}}}{|\dot{\theta}_{C_s^{-1}}|}  ds.
\]
Since the last integrand is bounded by one, we deduce that $\theta_t$ converges to a random point $\theta_{\infty}$ on the sphere as soon as $C_t$ converges when $t$ goes to infinity. Therefore, a sufficient condition for the almost sure convergence of the angle is the almost sure finiteness of $C_t$ as $t$ goes to infinity.

\begin{remark}\label{rem.convangle}
It is tempting to say that the finiteness of $C_t$ as $t$ goes to infinity is also a necessary condition for the almost sure convergence of the angle. For example, if $C_t$ goes to infinity and $\theta_t$ converges to $\theta_{\infty}$, in Equation \eqref{eqn.angular} describing the evolution of $\dot{\theta}_t/|\dot{\theta}_t|$, the first term of the right hand side would go to infinity in the direction of $\theta_{\infty}$ whereas the last terms would be asymptotically orthogonal to $\theta_{\infty}$ (see Remark \ref{rem.repmart}), which seems in contradiction with the boundedness of the left hand side. Unfortunately, we were not able to give a clear proof of this fact. Nevertheless,  the next proposition, which completes the proof of Proposition \ref{prop.angular}, 
gives a necessary and sufficient condition, in terms of the integrability of $f$, for the almost sure convergence of the clock $C_t$.
\end{remark}

\begin{prop}	\label{pro.convangle}
Suppose that the warping function $f$ is $\log-$concave and satisfies the integrability criterion $\int_1^{+\infty}f^{1-d}(\rho)d\rho <+\infty$, then the clock $C_t$ converges almost surely if and only if 
\[
I_d(f):=\int_1^{+\infty} f^{d-2}(r) \left( \int_r^{+\infty} f^{1-d}(\rho)d\rho \right) dr <+\infty.
\]
\end{prop}

\begin{proof}
Recall that, by Proposition \ref{pro.translog}, $r_t$ is transient almost surely if $f^{1-d}$ is integrable. 
By hypothesis, the warping function $f(r)$ goes to infinity with $r$. If it is $\log-$concave, the logarithmic derivative $f'/f(r)$ then admits a limit $\ell \geq 0$. Let us first suppose that $\ell>0$. In that case, $f$ has exponential growth and classical comparison results and Lemma \ref{lem.comparexpo} show that $r_t$ is ballistic. Therefore, the process $U_t$ converges almost surely as $t$ goes to infinity and the integral $I_d(f)$ is of course convergent, hence the result. Suppose now that $\ell=0$. Applying It\^o's formula, we have
\[
d \left( \frac{\sqrt{1-\dot{r}_t^2} }{f(r_t)} \right) =   \left( \frac{\sqrt{1-\dot{r}_t^2} }{f(r_t)} \right) \left[ - \frac{d-1}{2} \sigma^2 - 2 \frac{f'}{f}(r_t) \dot{r}_t \right] dt + \frac{\sigma^2}{2}   \left( \frac{(d-2)dt}{f(r_t)\sqrt{1-\dot{r}_t^2}} \right) - \sigma \frac{\dot{r}_t}{f(r_t)} dB_t,
\]
and 
\[
d \left( \frac{1-\dot{r}_t^2 }{f(r_t)} \right) =   \left( \frac{1-\dot{r}_t^2 }{f(r_t)} \right) \left[ - d \sigma^2 - 3 \frac{f'}{f}(r_t) \dot{r}_t \right] dt + \frac{\sigma^2}{2}   \left( \frac{(d-1)dt}{f(r_t)}\right) - 2\sigma \frac{\dot{r}_t \sqrt{1-\dot{r}_t^2}}{f(r_t)} dB_t,
\]
Since $r_t$ is transient and $f'/f(r)$ goes to zero as $r$ goes to infinity, from these two equations, we deduce the almost sure equivalences
\[
\int^{+\infty}\frac{1-\dot{r}_t^2}{f(r_t)} dt <+\infty \Longleftrightarrow \int^{+\infty}\frac{\sqrt{1-\dot{r}_t^2} }{f(r_t)} dt <+\infty \Longleftrightarrow \int^{+\infty}\frac{dt}{f(r_t)\sqrt{1-\dot{r}_t^2} }dt <+\infty.
\]
In other words, with the notations of the proof of Proposition \ref{pro.transneg}, the process $C_t$ converges almost surely as $t$ goes to infinity if and only if 
\[
\int^{+\infty}\frac{\sqrt{1-\dot{r}_t^2} }{f(r_t)} dt = \int^{+\infty}\frac{dt }{f(\rho_t)} <+\infty \;\; a.s.,
\]
where $\rho_t = r_{D^{-1}_t}$ and $D_t = \sigma^2 \int_0^t (1-\dot{r}_s^2)ds$. Otherwise, from Remark \ref{rem.horloge}, the clock $D_t$ is almost surely asymptotically linear in $t$, and from the proofs of Propositions \ref{pro.transneg} and \ref{pro.translog}, the process $\rho_t$ is asymptotically equivalent to 
the process $\widetilde{\sigma}^{-1} u_t$, where $u_t$ is solution of Equation (\ref{eqn.u}). In fine, we have the almost sure equivalence
\[
 \int^{+\infty}\frac{dt }{f(\rho_t)} <+\infty \Longleftrightarrow  \int^{+\infty}\frac{dt }{f(y_t)} <+\infty,
\]
where $y_t$ is the solution of the equation 
\[
d y_t = \frac{d-1}{2} \frac{f'(y_t)}{f(y_t)} dt +dB_t.
\]
The almost sure finiteness of the last integral is then equivalent to the fact that the new time changed diffusion $\widetilde{y}_t$ solution of 
  \[
d \widetilde{y}_t = \frac{d-1}{2} f'(\widetilde{y}_t)dt +\sqrt{f(\widetilde{y}_t)}dB_t,
\]
has an almost surely finite lifetime, which is finally equivalent to the finitness of $I_d(f)$, see e.g. Theorem 1.5 p. 212 of \cite{pinsky}.
\end{proof}

\subsection{Poisson boundary of the process}
\label{sec.poisson}

To conclude this section, we make explicit the Poisson boundary of kinetic Brownian motion in rotationally invariant manifolds whose warping function $f$ is $\log-$concave and for which $f^{1-d}$ is integrable at infinity. We show that the invariant sigma field of the whole kinetic Brownian motion $(z_t)_{t \geq 0}=(x_t, \dot{x}_t)_{t \geq 0}$ coincides almost surely with the sigma field generated by the escape angle when convergent. As announced above, we take advantage here of the powerfull d\'evissage method introduced in \cite{Devissage} that allows to compute the Poisson bourdary of a diffusion process starting from the one of a subdiffusion. Recall Theorem \ref{theo.devA} giving its main range of application. This method is particularly well suited here because, thanks to the warped product structure of rotationally invariant manifolds, the radial process is always a subdiffusion of dimension two of the whole process. The plan of the proof of Theorem \ref{theo.kpoisson} is thus the following : using coupling arguments, prove first that the Poisson boundary of the radial process $(r_t,\dot{r}_t)_{t \geq 0}$ is trivial, then use the d\'evissage method to deduce that the Poisson boundary of the full process is generated by the escape angle.

\subsubsection{Poisson boundary of the radial subdiffusion}

In this first section, we show that if $r_t$ is transient and $\dot{r}_t$ is Harris recurrent, then the Poisson boundary of the radial subdiffusion $(r_t,\dot{r}_t)_{t \geq 0}$ is trivial.
Let us first make a very simple but crucial remark: the first projection $r_t$ increases if and only if $\dot{r}_t$ is non negative so that the typical behaviour a of sample path of the radial subdiffusion looks like Fig. \ref{fig.poisson} below.\par
\begin{figure}[ht]
\begin{center}\hspace{0cm}
\includegraphics[scale=0.3]{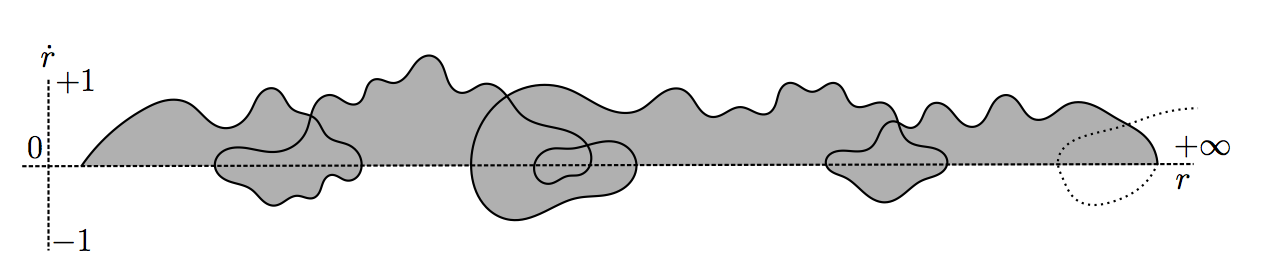}
\end{center}
\caption{Typical behaviour of the sample paths of the radial subdiffusion.\label{fig.poisson}}
\end{figure}
In this situation, we can exhibit a shift coupling between two independant copies of the radial process $(r_t,\dot{r}_t)_{t \geq 0}$. Recall that 
the existence of shift couplings between sample paths is equivalent to the fact that the underlying process has trivial Poisson boundary, see e.g. \cite{cranston}.

\begin{prop}\label{pro.poissontrivradial}
If $r_t$ is transient almost surely and if $\dot{r}_t$ is Harris recurrent in $(-1,+1)$, then the Poisson boundary of the radial subdiffusion $(r_t, \dot{r}_t)$ is trivial.
\end{prop}

\begin{proof}
Let $(r^1_t, \dot{r}_t^1)$ and $(r^2_t, \dot{r}_t^2)$ be two independant solutions of Equation (\ref{eqn.radial})  starting from $(r_0^1, \dot{r}_0^1)$ and  $(r_0^2, \dot{r}_0^2)$ respectively. 
Without lost of generality, we can suppose that $r_0^2 \leq r_0^1$.
Let us define $S:=\inf \{ t \geq 0, \, \dot{r}_t^1=1/2\}$ and $T:=\inf \{ t \geq S, \, \dot{r}_t^1=0\}$. Since $\dot{r}_t^1$ is Harris recurrent, $S$ and $T$ are finite almost surely.
Since $r_t^1$ increases if and only if $\dot{r}_t^1$ is non negative, we have almost surely
\[
\inf_{r \geq r_0^1} \alpha(r) \geq 0, \;\; \hbox{where} \;\; \alpha(r):=\sup_{ t \geq 0} \{\dot{r}_t^1, \, r_t^1=r\} \geq 0.
\] 
Geometrically, this means that the top of the grey zone associated with $(r^1_t, \dot{r}_t^1)$ on Fig. \ref{fig.poisson2} below is always non negative.
Since $r_t^2$ goes almost surely to infinity with $t$, the hitting time $T':=\inf \{ t \geq 0, \, r_t^2=r_T^1\}$ is finite almost surely. If $\dot{r}^2_T > \alpha(r_T^1)$, since $\dot{r}^2_t$ is recurrent, the curves $(r^1_t, \dot{r}_t^1)$ and $(r^2_t, \dot{r}_t^2)$ must intersect in finite time. Indeed, $\dot{r}^2_t$ must visit zero after $T'$ and the grey zone is a barrier, this situation is illustrated in red on Fig. \ref{fig.poisson2} below. If $\dot{r}^2_T \leq \alpha(r_T^1)$, then the two curves must have intersected before $T$, this situation is illustrated in green on Fig. \ref{fig.poisson2}. 
\end{proof}

\begin{figure}[ht]
\hspace{2cm}\begin{center}
\includegraphics[scale=0.35]{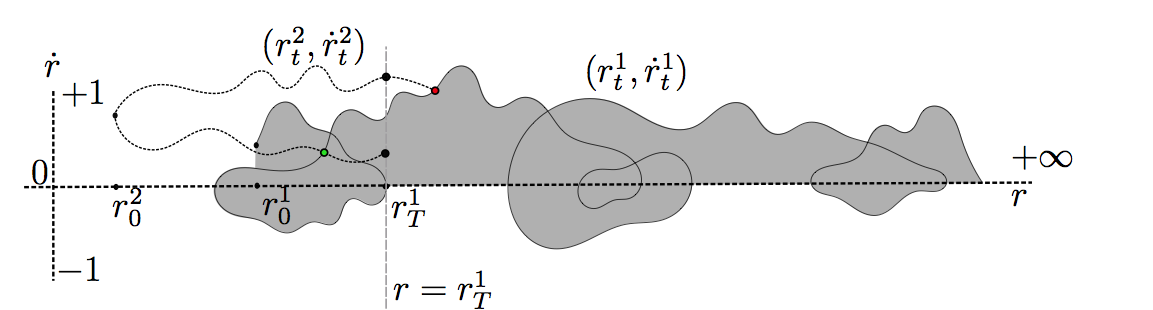}
\end{center}
\caption{Typical behaviour of the sample paths of the radial subdiffusion.\label{fig.poisson2}}
\end{figure}

\begin{coro}\label{cor.poissonradial}
Suppose that the warping function $f$ is $\log-$concave and satisfies the integrability criterion $\int_1^{+\infty}f^{1-d}(\rho)d\rho <+\infty$, then the Poisson boundary of the radial subdiffusion $(r_t, \dot{r}_t)$ is trivial.
\end{coro}

\begin{proof}
From Proposition \ref{pro.translog}, if $f^{1-d}$ is integrable at infinity, the process $r_t$ is transient almost surely. Otherwise, if $f$ is $\log-$concave, then 
$f'/f$ is bounded ouside a compact set and from Proposition \ref{pro.recrpoint}, $\dot{r}_t$ is then Harris recurrent in $(-1,+1)$. Thus the assumptions of Proposition \ref{pro.poissontrivradial} are fulfilled, hence the result.
\end{proof}

\subsubsection{A diffusion lifting kinetic Brownian motion}

Let us first observe that one cannot directly use the devissage method to make explicit the Poisson boundary of the  kinetic Brownian motion $(r_t, \dot{r}_t, \theta_t, \dot{\theta}_t / \vert \dot{\theta}_t \vert)_{t \geq 0}$. Indeed, even if $\theta_t$ converges almost surely to $\theta_{\infty}$ when $t$ goes infinity, the dynamics of $ \dot{\theta}_t / \vert \dot{\theta}_t \vert$ involves terms in $\theta_t$ (to keep $\theta_t$ on the sphere $\mathbb{S}^{d-1})$ and the process $(r_t,\dot{r}_t,\dot{\theta}_t / \vert \dot{\theta}_t \vert)_{t\geq 0}$ is not a subdiffusion  although we do not expect invariant information from this term. In particular, the d\'evissage condition of Theorem \ref{theo.devA} is not fulfilled.
In order to perform the setting of application of the d\'evissage method we thus need to represent the angular process $(\dot{\theta}_t / \vert \dot{\theta}_t \vert, \theta_t)_{t \geq 0} \in T^1 \mathbb{S}^{d-1}$ by a process with values in a bigger space where the d\'evissage method can indeed be applied. Namely, in two steps, we will represent the angular process as a projection of some process $ (b_t, g_t)_{t \geq 0}$ living in $SO(d-1) \times SO(d)$. Recall that $(\epsilon_1, \dots, \epsilon_d)$ denotes the canonical basis of $\mathbb{R}^d$ and let us denote by $SO(d-1)$ (resp.  $SO(d-2)$)  the subgroup of $SO(d)$ made of elements fixing $\epsilon_1$ (resp. fixing both $\epsilon_1$ and $\epsilon_2$).

First remark that the angular process $(\theta_t, \dot{\theta}_t /|\dot{\theta}_t|)_{t \geq 0} $ takes values in $T^1 \mathbb S^{d-1}$, which can be seen as the homogeneous space $SO(d)/SO(d-2)$. Then consider the process $(r_t, \dot{r}_t, e_t)_{t \geq 0}$, where $e_t=(e^1_t, \ldots, e^d_t) \in SO(d)$, whose infinitesimal generator is given by
\begin{equation}\label{eqn.genlift}
\mathcal G= \mathcal L_{(r, \dot{r})} + \alpha(r, \dot{r}) H_0  + \frac{\sigma^2}{2}  \beta(\dot{r}) \sum_{j=3}^d V_j^2,
\end{equation}
where
\[
\alpha(r, \dot{r}) :=\frac{\sqrt{1-\dot{r}^2}}{f(r)}, \qquad \beta(\dot{r}):= \frac{1}{1-\dot{r}^2} 
\]
and $H_0(g):=g {H}_0$ (resp. $V_j(g) = g{  V}_j $)  is the left invariant vector field in $SO(d)$ generated by ${ H}_0:= \epsilon_1 \otimes \epsilon_2 -\epsilon_2 \otimes \epsilon_1$ (resp. ${V}_j:= \epsilon_2 \otimes \epsilon_j-\epsilon_j \otimes \epsilon_2$). In other words, the process $(e_t)_{t \geq 0}$ is solution of the Stratonovich stochastic differential equation 
\begin{equation}\label{enq.e}
d e_t = \alpha(r_t, \dot r_t)   e_t { H}_0   dt    + \beta(\dot{r}_t) \sum_{j=3}^d e_t {V}_j \circ dB_t^j,
\end{equation}
for a standard $(d-2)$-Euclidean Brownian motion $(B^3_t, \ldots, B^d_t)$. Doing so,
the process $(r_t, \dot{r}_t, e_t(\epsilon_1), e_t(\epsilon_2))_{t \geq 0}$ has the same law as the target process $\big(r_t, \dot{r}_t, \theta_t, \dot{\theta}_t /|\dot{\theta}_t|\big)_{t \geq 0}$ i.e. in this first step, we lifted kinetic Brownian motion to a diffusion on $\R^+ \times [-1,1] \times SO(d)$.
Let us now consider the time changed Brownian motion $(b_t)_{t \geq 0}$ on $SO(d-1)$ solution of 
\begin{equation}
d b_t = \beta (\dot r_t) \sum_{i=3}^d b_t { V}_i \circ dB_t^i. \label{bt}
\end{equation}
Finally, viewing $b_t$ as the element $\textrm{diag}(1,b_t)$ of $SO(d)$,  define the process $(g_t)_{t \geq 0}$ on $SO(d)$ by $g_t:=e_t b_t^{-1}$ so that we have
\begin{equation}
d g_t = \alpha(r_t, \dot r_t)  g_t  \,  b_t \, { H}_0 \, b_t^{-1} dt. \label{gt}
\end{equation}
Note that, by construction, we have $g_t(\epsilon_1)=\theta_t$ and $g_t b_t(\epsilon_2) = \dot{\theta_t}/ \vert \dot{\theta_t} \vert$, i.e. after this second step, kinetic Brownian motion is now represented by a diffusion $(r_t, \dot{r}_t, b_t, g_t)_{t \geq 0}$ with values in $\R^+ \times [-1,1] \times SO(d-1)\times SO(d)$.
But, now we are in good position to apply the d\'evissage scheme. Indeed, the process $(r_t, \dot{r}_t, b_t)_{t \geq 0}$ is a now subdiffusion of $(r_t, \dot{r}_t, b_t, g_t)_{t \geq 0}$. Moreover one shows that this diffusion has a trivial Poisson boundary.

\begin{prop}\label{pro.sousdiff}
Suppose that the warping function $f$ is $\log-$concave and satisfies the integrability criterion $\int_1^{+\infty}f^{1-d}(\rho)d\rho <+\infty$, the Poisson boundary of the subdiffusion $(r_t, \dot{r}_t,b_t)_{t \geq 0}$ is trivial.
\end{prop}

\begin{proof}
Let us first remark that, following \cite{cranston}, it is sufficient to construct a shift coupling between any two $\big(r_t,\dot r_t,b_t\big)$-diffusions started from two different initial conditions to get the triviality of the Poisson boundary. We can construct such a shift coupling in the following way.
Let $\big(r_t^1, \dot{r}_t^1, b^1_t\big)_{t \geq 0}$ and $\big(r_t^2, \dot{r}_t^2, b^2_t\big)_{t \geq 0}$ be two independant versions of the $\big(r_t,\dot r_t,b_t\big)$-diffusion starting from 
$\big(r_0^1, \dot{r}_0^1, b_0^1\big)$ and $\big(r_0^2, \dot{r}_0^2, b_0^2\big)$ respectively.  We know from Proposition \ref{pro.poissontrivradial} that there exists random times $T_1,T_2$ that are finite $\PP_{(r_0^1, \dot{r}_0^1, b^1_0)}$-almost surely, respectively $\PP_{(r_0^2, \dot{r}_0^2, b^2_0)}$-almost surely, and such that one can modify the process $\big(r_t^1, \dot{r}_t^1, b^1_t\big)_{t \geq 0}$ so as to have $\big(r_{T_1+t}^1, \dot{r}_{T_1+t}^1\big) = \big(r_{T_2+t}^2, \dot{r}_{T_2+t}^2\big)$, for all $t \geq 0$. In particular, for all $t \geq 0$, we have 
\[
\int_{T_1}^{T_1+t}\frac{ds}{1-|\dot r_s^1|^2} = \int_{T_2}^{T_2+t}\frac{ds}{1-|\dot r_s^2|^2}.
\]
Since Brownian motion in $SO(d-1)$ is ergodic, one can find two Brownian motions $(c^1(t))_{t\geq 0}$ and $(c^2(t))_{t\geq 0}$ on $SO(d-1)$, started from $b^1_{T_1}$ and $b^2_{T_2}$ respectively, which are independant of the radial subdiffusions and which couple almost surely in finite time. Then, the processes $y^i_t$, $i=1,2$, defined by the formulas
\begin{equation*}
y^i_t = \left \lbrace \begin{array}{ll}
\displaystyle{\big(r_t^i, \dot{r}_t^i, b^i_t\big)}, & \textrm{for } 0\leq t\leq T_i, \\
\\
\displaystyle{ \left(r_t^i, \dot{r}_t^i, c^i \left( \int_{T_i}^t\frac{ds}{1-|\dot r_s^i|^2}\right)\right)}, & \textrm{for }t\geq T_i,
\end{array}\right.
\end{equation*}
have the laws of $\big(r_t,\dot r_t,b_t\big)$-diffusions started from $\big(r_0^1, \dot{r}_0^1, b^1_0\big)$ and $\big(r_0^2, \dot{r}_0^2, b^2_0\big)$, respectively and will couple in finite time almost surely, hence the result.
\end{proof}

\subsubsection{Poisson boundary of the full diffusion}

We can now give the proof of Theorem \ref{theo.kpoisson} describing the Poisson boundary of the full kinetic Brownian motion. As mentioned before, the key tool used in this proof is the \textit{d\'evissage method} introduced in \cite{Devissage}, with the twist that we shall not apply it to kinetic Brownian motion, but rather to the lifted diffusion $( r_{t}, \dot{r}_{t}, g_{t}, b_{t})_{t \geq 0}$ with values $\R_+\times [-1,1]\times SO(d)\times SO(d-1)$ introduced in the last paragraph. In a final step one obtains the Poisson boundary of the kinetic Brownian motion from the Poisson boundary of the this lift.  

\medskip

We have shown in Proposition \ref{pro.sousdiff} that if $f$ is $\log-$concave and satisfies the integrability criterion $\int_1^{+\infty}f^{1-d}(\rho)d\rho <+\infty$, then the Poisson boundary of the subdiffusion $(r_t, \dot{r}_t,b_t)_{t \geq 0}$ is trivial. Moreover, if the second integrability condition $I_d(f)<\infty$ of Proposition \ref{pro.convangle} is satisfied, the clock $C_t$ is almost surely convergent and from Equation \eqref{gt}, the component $g_{t}$ then converges almost surely to $g_{\infty}$ in $SO(d)$.  In other words, the d\'evissage and convergence conditions of Theorem \ref{theo.devA} are satified, with $\mathcal{N}= \mathbb{R}_+ \times [-1,1] \times SO(d-1)$, $G=SO(d)$, with $\big(r_t, \dot{r}_t, b_t\big)_{t \geq 0}$ in the role of the subdiffusion. Moreover, by equations \eqref{bt} and \eqref{gt}, the infinitesimal generator $\mathcal{G}_b$ of the full diffusion $( r_{t}, \dot{r}_{t}, g_{t}, b_{t})_{t \geq 0}$ is very similar to the generator $\mathcal G$ defined in \eqref{eqn.genlift}, except the geodesic flow is now conjugated, namely 
 \[
\mathcal{G}_b:= \mathcal L_{(r, \dot{r})} +  \beta(\dot{r}) \sum_{i=3}^{d} { V}_j^2  + \alpha(r, \dot{r}) H_b,
 \]
where $H_b$ denotes the left invariant vector field on $SO(d)$ defined at point $g$ by the formula $g b H_0 b^{-1} \in T_g SO(d)$. It follows that the equivariance condition \textit{(3)} of theorem \ref{theo.devA} is fulfilled, while it is easy to check that $\mathcal{G}_b$ satisfies H\"ormander condition and, so, is hypoelliptic. As the regularity condition is also fulfilled, the d\'evissage method applies. Namely, Theorem \ref{theo.devA} implies the following statement:
 
 \begin{prop}\label{poisson.en.haut}
 Suppose that the warping function $f$ is $\log$-concave and satisfies
 \[
I_d(f):= \int_1^{+\infty} f^{d-2}(r) \left ( \int_r^{+\infty} f^{1-d}(\rho)d\rho \right ) dr < +\infty.
 \]
Then, the Poisson boundary of the diffusion $(r_t,\dot{r}_t, b_t,g_t)_{t \geq 0}$ coincides almost surely with $\sigma(g_\infty)$.
 \end{prop}
 
To finally deduce the Poisson boundary of kinetic Brownian motion starting from the Poisson boundary of its lift $(r_t,\dot{r}_t, b_t,g_t)_{t \geq 0}$ with values in $\R_+\times [-1,1]\times SO(d)\times SO(d-1)$, it is convenient to introduce the map $\chi$ from $\R^+ \times [-1,1] \times SO(d) \times SO(d-1)$ to the base space $\R^+ \times [-1,1] \times  T^1\mathbb{S}^{d-1}$, given by the formula 
 \[
\chi\big(r, \dot{r}, g,b\big) = \big(r, \dot{r}, g(\epsilon_1),gb(\epsilon_2) \big).
 \]
Let now $f : \R^+ \times [-1,1] \times  T^1\mathbb{S}^{d-1} \to \R$ be a bounded harmonic function for the infinitesimal generator of kinetic Brownian motion $(r_t, \dot{r}_t, \theta_t, \dot{\theta}_t/|\dot{\theta}_t|)_{t \geq 0}$. Then, $f \circ \chi$ is $\mathcal{G}_b$-harmonic, and by Proposition \ref{poisson.en.haut} there exists a bounded measurable real-valued function $\textrm{F}$ on $SO(d)$ such that we have
 \[
(f \circ \chi)\big( r, \dot{r}, g, b\big) = \mathbb{E}_{(r,\dot{r}, g, b)}\big[\textrm{F}(g_\infty)\big],
 \] 
for all $(r,\dot{r}, g, b)$. Let now fix $\big(r,\dot{r}, \theta,  \dot{\theta}/ \vert \dot{\theta} \vert\big)$ and $g \in SO(d)$, such that 
 \[
\big(r,\dot{r}, \theta, \dot{\theta} / \vert \dot{\theta} \vert\big) = \chi \big(r, \dot{r}, g, \mathrm{Id}\big),
 \] 
 and remark that we have $\chi \big(r, \dot{r}, g, \mathrm{Id}\big) = \chi \big(r, \dot{r}, gk, k^{-1}\big)$, for any $k\in SO(d-1)$, so one can write
 \begin{align}
 f\big(r,\dot{r}, \theta,  \dot{\theta} / \vert \dot{\theta} \vert \big) = \mathbb{E}_{(r,\dot{r}, gk, k^{-1})} \big[ \textrm{F}(g_\infty) \big], \label{forallk}
 \end{align}
 for all $k\in SO(d-1)$. The conclusion will then come from the following equivariance result.
 \begin{lemma}
 The law of $(g_t)_{t \geq 0}$ under $\mathbb{P}_{(r,\dot{r}, gk,k^{-1})}$ coincides with the law of $\big(g_t k\big)_{t \geq 0}$ under $\mathbb{P}_{(r, \dot{r}, g, \mathrm{Id})}$, for all $k\in SO(d-1)$. 
 \end{lemma}
 
 \begin{proof}
 Consider the processes $(g_t,b_t)_{t \geq 0}$ starting at $(gk, k^{-1})$, and set $g'_t:= g_t k $ and $b'_t:= kb_t$. Then, using  equations \eqref{bt} and \eqref{gt}, one easily checks that the process $b'_t$ starts from the identity and satisfies the equation
 \begin{align*}
 db'_t &= \beta(\dot{r}_t) \sum_{i=3}^d b'_t { V}_i \circ dB_t^i
 \end{align*}
 while the process $(g'_t)_{t \geq 0}$ starts from $g$ and satisfies the equation
 \begin{align*}
 d g'_t &= \alpha(r_t, \dot{r}_t)  g'_t b'_t {H}_0 (b'_t)^{-1} dt.
 \end{align*}
 So the law of the process $\big(r_t, \dot{r}_t, g'_t,b'_t\big)_{t \geq 0}$ is $\mathbb{P}_{(r,\dot{r}, g,\mathrm{Id})}$, giving the result. 
 \end{proof}
 
 \smallskip
 
It follows from the previous lemma that the law of $g_\infty$ under $\mathbb{P}_{(r,\dot r, gk, k^{-1})}$ coincides with the law of $g_\infty k$ under $\mathbb{P}_{(r,\dot r, g,\mathrm{Id})}$. However, as 
\begin{align}
f\big(r, \dot{r},\theta,  \dot{\theta}/ \vert \dot{\theta} \vert \big) = \mathbb{E}_{(r, \dot{r}, g, \mathrm{Id})}\big[ \textrm{F}(g_{\infty} k)\big],
\end{align} 
for all $k \in SO(d-1)$, from equation \eqref{forallk}, we have
\begin{align}
f\big(r, \dot{r}, \theta,  \dot{\theta}/ \vert \dot{\theta} \vert \big) = \mathbb{E}_{(r,\dot{r},g, \mathrm{Id})} \left [ \int_{SO(d-1)} \textrm{F}(g_{\infty} k)\,dk\right],
\end{align}
where $dk$ stands for the normalized Haar measure on $SO(d-1)$. Now, given any measurable section $S : \mathbb{S}^{d-1}\simeq SO(d)/SO(d-1) \to SO(d)$ to the natural projection, we define a bounded measurable real-valued function $\overline{\textrm{F}}$ setting
\[
\overline{\textrm{F}}(\theta) := \int_{SO(d-1)} \textrm{F}\big(S(\theta)k\big)\,dk.
\]
In those terms, we finally have
\begin{align*}
f\big(r,\dot{r}, \theta, \dot{\theta}/ \vert \dot{\theta} \vert\big) = \mathbb{E}_{(r,\dot{r}, g,\mathrm{Id})} \big[\,\overline{\textrm{F}}\big(g_\infty(\epsilon_1)\big)\Big] = \mathbb{E}_{(r,\dot{r},\dot{\theta}/ \vert \dot{\theta} \vert, \theta)} \Big[\,\overline{\textrm{F}}\big(\theta_{\infty}\big)\Big],
\end{align*}
which shows indeed that the Poisson boundary of kinetic Brownian motion is generated by the escape angle $\theta_\infty$.

\appendix

\section{Kinetic Brownian motion in the hyperbolic plane}
\label{Section2DimHypSpace}
In the preceeding section, we exhibited the almost sure long time asymptotic behavior of the kinetic Brownian motion in rotationally invariant manifolds of dimension $d \geq 3$. As noted in Remark \ref{rem.d2}, the case $d=2$ requires a separate treatment since in that case the norm $|\dot{\theta}_t|$ vanishes almost surely in finite time so that the process $\dot{\theta}_t/|\dot{\theta}_t|$ is not well defined. Nevertheless, this is just a formal difficulty due to the choice of polar coordinates, and an exhaustive study can be done in the same lines as above on a rotationally invariant manifold of dimension 2.  
To emphasize this fact, let us consider the case of $\mathbb H^2$, the hyperbolic space of dimension 2. 
\subsection{Kinetic Brownian motion in $T^1 \mathbb H^2$}
We consider the half-space model of the hyperbolic plane, namely we view $\mathbb H^2$ as the half-space $\{(x,y) \in \mathbb R^2, \, y>0\}$ endowed with the metric $ds^2=y^{-2} (dx^2+dy^2)$. In this coordinates system, the kinetic Brownian motion $(x_t, y_t, \dot{x}_t, \dot{y}_t)$ started at $(x_0, y_0, \dot{x}_0, \dot{y}_0) \in T^1 \mathbb H^2$ is the solution of the system of stochastic differential equations \eqref{eqn.system} which simply writes here:
\begin{equation}\label{eqn.h2}
\left \lbrace
\begin{array}{ll}
d x_t & = \dot{x}_t dt, \\
\\
d y_t & = \dot{y}_t dt, \\
\\
d \dot{x}_t & = \displaystyle{2 \frac{\dot{x}_t \dot{y}_t}{y_t}dt -\frac{\sigma^2}{2} \dot{x}_t dt + \sigma d M^{\dot{x}}_t}, \\
\\
d \dot{y}_t & = \displaystyle{\left(  \frac{\dot{y}_t^2}{y_t}-\frac{\dot{x}_t^2}{y_t}\right) dt -\frac{\sigma^2}{2} \dot{y}_t dt + \sigma d M^{\dot{y}}_t},
\end{array}
\right.
\end{equation} 
where the martingales $M^{\dot{x}}$ and $M^{\dot{y}}$ have the following covariations
\[
d \langle M^{\dot{x}}, \,M^{\dot{x}} \rangle_t = \left( y^2_t - \dot{x}_t^2\right) dt, \quad d \langle M^{\dot{x}}, \,M^{\dot{y}} \rangle_t = -  
\dot{x}_t \dot{y}_t dt, \quad d \langle M^{\dot{y}}, \,M^{\dot{y}} \rangle_t = \left( y^2_t - \dot{y}_t^2\right) dt,
\]
and can thus be represented by a real Brownian motion $(B_{t})_{t \geq 0}$: 
\[
d M^{\dot{x}}_t = \dot{y}_t d B_t, \qquad d M^{\dot{y}}_t = - \dot{x}_t d B_t.  
\]
Recall moreover that the process is parametrized by the arc length so that we have almost surely, for all time $t \geq 0$:
\begin{equation}\label{eqn.metrich2}
\displaystyle{\frac{\dot{x}_t^2+ \dot{y}_t^2}{y_t^2} = 1}.
\end{equation}

Note that since the process can not explode in finite time, we have $y_t \in (0,+\infty)$ for all $t \geq 0$ almost surely. 
As in the case where $d\geq 3$, the study of the process is then made easier thanks to the presence of two subdiffusions:
\begin{lemma}
The kinetic Brownian motion  $(x_t, y_t, \dot{x}_t, \dot{y}_t)$ in $T^1\mathbb H^2$ admits the following subdiffusions of dimension one, two and three respectively:
\[
(u_t)_{t \geq 0}:= \left( \frac{\dot{y}_t}{y_t} \right)_{t \geq 0}, \qquad (y_t, u_t)_{t \geq 0}, \qquad (y_t, \dot{x}_t, \dot{y}_t)_{t \geq 0}.
\]
\end{lemma}
\begin{proof}
Starting from the system \eqref{eqn.h2}, a direct application of It\^o's formula gives
\[
\left \lbrace
\begin{array}{ll}
d y_t = u_t y_t dt, \\
\\
d u_t = \displaystyle{-(1-u_t^2) dt -\frac{\sigma^2}{2} u_t dt + \sigma \sqrt{1-u_t^2} d B_t,}
\end{array}
\right.
\]
hence the result for the first two subdiffusions. Moreover, from Equation \eqref{eqn.h2}, the evolution of $( y_t, \dot{x}_t, \dot{y}_t)_{t \geq 0}$ does not depend on $x_t$.
\end{proof}

\begin{remark}\label{rem.sym}
Note that, by symmetry if $(y_t, \dot{x}_t, \dot{y}_t)$ is a version of the $3-$dimensional subdiffusion starting from $(y_0, \dot{x}_0=0, \dot{y}_0=\pm 1)$, then $(y_t, -\dot{x}_t, \dot{y}_t)$ is also a version of the process.
\end{remark}

\subsection{Long time asymptotics}

From the ergodicity of the subdiffusion $(u_t)_{t \geq 0}$, one can then deduce the following almost sure long time asymptotic behavior for the whole process:

\begin{prop}
Let $(x_t, y_t, \dot{x}_t, \dot{y}_t)$ be the kinetic Brownian motion in $T^1 \mathbb H^2$ starting from $(x_0, y_0, \dot{x}_0, \dot{y}_0)$. Then the process $(u_t)_{t \geq 0}$ is ergodic in $[-1,1]$ with invariant probability measure 
\[
\mu(dx) = \displaystyle{ \frac{1}{Z} \frac{e^{-2x/\sigma^2}}{\sqrt{1-x^2}}},
\]
where $Z$ is a normalizing constant. Moreover, as $t$ goes to infinity, we have 
\[
 \lim_{t \to +\infty} \frac{1}{t} \log \left( \frac{y_t}{y_0} \right) = \int_{-1}^1 x \mu(dx) \in (-\infty, 0),
\]
in particular $y_t$ goes almost surely to zero exponentially fast and the process $x_t$ converges almost surely to a random variable $x_{\infty} \in \mathbb R$.
\end{prop}

\begin{figure}[ht]
\begin{center}\includegraphics[scale=0.55]{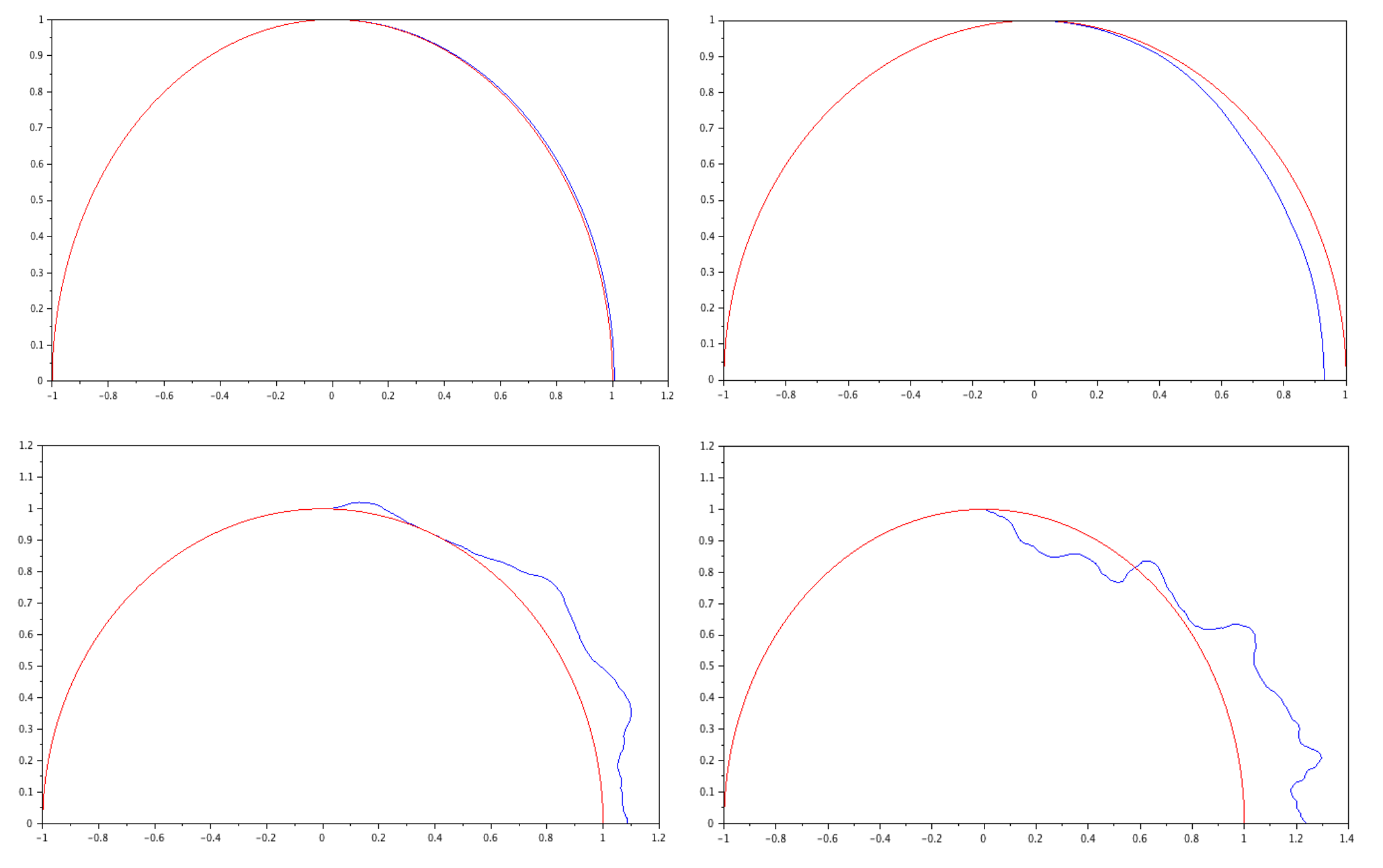}\end{center}
\caption{Simulation of the process and the corresponding geodesic for $\sigma=0.01, 0.1, 1, 2$.}
\end{figure}

\begin{proof}
One easily checks that $\mu$ is an invariant probability measure for $(u_t)_{t \geq 0}$ which is thus ergodic. The ergodic theorem then ensures that 
almost surely
\[
\frac{1}{t} \log \left( \frac{y_t}{y_0} \right) = \frac{1}{t} \int_0^t u_s ds \xrightarrow{t \to +\infty} \int_{-1}^1 x \mu(dx) \in (-\infty, 0).
\]
In particular, almost surely, $y_t$ goes to zero exponentially fast as $t$ goes to infinity. Moreover, from the relation \eqref{eqn.metrich2}, we have $\dot{x}_t^2 \leq y_t^2$ from which we deduce that 
\[
x_t = x_0 =\int_0^t \dot{x}_s ds
\]
is also almost surely convergent, hence the result.
\end{proof}

\subsection{Poisson boundary}
We conclude the study of the kinetic Brownian motion in the hyperbolic plane by expliciting its Poisson boundary. We show that the invariant sigma field of the full process coincides almost surely the sigma field generated by the point $(x_{\infty},0) \in \partial \mathbb H^2$ the boundary at infinity.

\begin{prop}
Let $(x_t, y_t, \dot{x}_t, \dot{y}_t)$ be the kinetic Brownian motion in $T^1 \mathbb H^2$ starting from $(x_0, y_0, \dot{x}_0, \dot{y}_0)$. Then the Poisson boundary of the process coincides almost surely with $\sigma(x_{\infty})$.
\end{prop}

\begin{proof}
The proof is again an application of the d\'evissage method. Since $(u_t)_{t \geq 0}$ is ergodic in $[-1,1]$, its invariant sigma field is clearly trivial. 
Otherwise, the invariant sigma field of the process $(y_t, \dot{y}_t)_{t \geq 0}$ coincides almost surely with the one of the process $(-\log(y_t),u_t)_{t \geq 0}$, whose behavior is stricly similar to the one of the process $(r_t, \dot{r}_t)_{t \geq 0}$ in the proof of Proposition \ref{pro.poissontrivradial}, namely $-\log(y_t)$ goes almost surely to infinity with $t$ while its derivative $u_t$ is ergodic. The exact same proof then shows that the invariant sigma field of $(-\log(y_t),u_t)_{t \geq 0}$ and thus $(y_t, \dot{y}_t)_{t \geq 0}$ is trivial. Equivalently, if $(y_t^1, \dot{y}_t^1)_{t \geq 0}$ and $(y_t^2, \dot{y}_t^2)_{t \geq 0}$ are two versions of the process starting from $(y_0^1, \dot{y}_0^1)$ and $(y_0^2, \dot{y}_0^2)$ respectively, there exists two random times $T_1$ and $T_2$ that are finite almost surely and such that $(y_{T_1+t}^1, \dot{y}_{T_1 +t}^1)_{t \geq 0}=(y_{T_2+t}^2, \dot{y}_{T_2+t}^2)_{t \geq 0}$. 
Without loss of generality, we can suppose that $T_1 \geq S$ where $S$ is the first hitting point of $\pm 1$ by $u^1_{t} = \dot{y}_{t}^1 /y_{t}^1$, which is also finite almost surely since $u_1$ ergodic in $[-1,1]$. 
Now, the relation \eqref{eqn.metrich2} implies that a shift coupling between $(y_{t}^1, \dot{y}_{t}^1)$ and  $(y_{t}^2, \dot{y}_{t}^2)$ yields automatically a shift coupling between the corresponding versions of the process $(y_t, \dot{x}_t, \dot{y}_t,)_{t \geq 0}$. Indeed, Equation \eqref{eqn.metrich2} then ensures that $|\dot{x}_{T_1}^1|^2 = |\dot{x}_{T_2}^2|^2$ almost surely. If $\dot{x}_{T_1}^1 = \dot{x}_{T_2}^2$, there is nothing to prove, if $\dot{x}_{T_1}^1 = -\dot{x}_{T_2}^2$, we can use the symmetry argument of Remark \ref{rem.sym} and consider the process $(-\dot{x}^1_{S+t})_{t \geq 0}$ starting at $-\dot{x}^1_{S}=0$. In other words, the invariant sigma field of the subdiffusion $(y_t, \dot{y}_t,\dot{x}_t)_{t \geq 0}$ is almost surely trivial. To apply the devissage scheme, we are left to check the equivariance and regularity properties of Theorem \ref{theo.devA}.
We have $x_t=x_0 +\int_0^t \dot{x}_s ds$ so that the infinitesimal generator of the full diffusion $(x_t, y_t, \dot{x}_t, \dot{y}_t)$ is clearly equivariant under the action of $\mathbb R$ by translation on the $x$ component. Moreover, the generator is hypoelliptic so that harmonic function are continuous. 
By Theorem \ref{theo.devA}, we can conclude that the invariant sigma field of the kinetic Brownian motion $(x_t, y_t, \dot{x}_t, \dot{y}_t)$ coincides almost surely with the sigma field generated by the limit $x_{\infty}$.
\end{proof}

\if{
\newpage
\section{Evolution of the normalized angular derivative}\label{app.B}
We give in this Appendix the details of the computations that are needed to give a complete proof of proposition 3.3.2.
Starting from Equation \eqref{eqn.system.rot}, we have
\begin{equation}\label{eqn.ap1}
d f^2(r_t) \dot{\theta}^i_t = - \frac{(d-1)\sigma^2}{2} f^2(r_t) \dot{\theta}^i_t dt + \sigma f^2(r_t) d M^{ \dot{\theta}^i}_t.
\end{equation}
Taking the square, we get 
\[
d f^4(r_t) |\dot{\theta}^i_t |^2= - d \sigma^2 f^4(r_t) |\dot{\theta}^i_t|^2 dt + \sigma^2 f^2(r_t) dt + 2 \sigma f^4(r_t) \dot{\theta}^i_t d M^{ \dot{\theta}^i}_t,
\]
and making the sum over $1 \leq i \leq d$:  
\[
d f^4(r_t) |\dot{\theta}_t |^2= - d \sigma^2 f^4(r_t) |\dot{\theta}_t|^2 dt + d \sigma^2 f^2(r_t) dt + 2 \sigma f^4(r_t) \sum_{i=1}^d \dot{\theta}^i_t d M^{ \dot{\theta}^i}_t.
\]
Taking the square root, one then deduce that 
\begin{equation}\label{eqn.ap2}
d  f^2(r_t) |\dot{\theta}_t | =  - \frac{(d-1)\sigma^2}{2} f^2(r_t) |\dot{\theta}_t| dt + \frac{(d-1)\sigma^2}{2|\dot{\theta}_t|} dt +  \sigma f^2(r_t) \sum_{i=1}^d \frac{\dot{\theta}^i_t}{|\dot{\theta}_t|} d M^{ \dot{\theta}^i}_t.
\end{equation}
From Equations \eqref{eqn.ap1} and \eqref{eqn.ap2}, taking the quotient, and using the rule 
\[
 d \frac{X_t}{Y_t} = \frac{1}{Y_t} d X_t  -\frac{X_t}{Y_t^2} d Y_t + \frac{X_t}{Y_t^3} d \langle Y_t\rangle - \frac{1}{Y_t^2} d \langle X_t, Y_t \rangle
\]
we have thus
\begin{equation}\label{eqn.ap3}
\begin{array}{ll}
d \displaystyle{\frac{\dot{\theta}^i_t}{|\dot{\theta}_t|} }=  &  \displaystyle{\frac{1}{f^2(r_t) |\dot{\theta}_t |} \left[ - \frac{(d-1)\sigma^2}{2} f^2(r_t) \dot{\theta}^i_t dt + \sigma f^2(r_t) d M^{ \dot{\theta}^i}_t \right]} \\
\\
&  \displaystyle{- \frac{f^2(r_t) \dot{\theta}_t^i }{f^4(r_t) |\dot{\theta}_t |^2} \left[- \frac{(d-1)\sigma^2}{2} f^2(r_t) |\dot{\theta}_t| dt + \frac{(d-1)\sigma^2}{2|\dot{\theta}_t|} dt +  \sigma f^2(r_t) \sum_{j=1}^d \frac{\dot{\theta}^j_t}{|\dot{\theta}_t|} d M^{ \dot{\theta}^j}_t \right]} \\
\\
&  \displaystyle{+\frac{f^2(r_t) \dot{\theta}_t^i }{f^6(r_t) |\dot{\theta}_t |^3} \left[ \sigma^2 f^4(r_t) \left \langle \sum_{j=1}^d \frac{\dot{\theta}^j_t}{|\dot{\theta}_t|} d M^{ \dot{\theta}^j}_t   \right \rangle  \right]  } \\
\\
&  \displaystyle{ - \frac{\sigma^2}{f^4(r_t) |\dot{\theta}_t |^2} \left[ \left \langle f^2(r_t) d M^{ \dot{\theta}^i}_t,  f^2(r_t) \sum_{j=1}^d \frac{\dot{\theta}^j_t}{|\dot{\theta}_t|} d M^{ \dot{\theta}^j}_t   \right \rangle \right].}
\end{array}
\end{equation}
Otherwise, we have 
\[
\begin{array}{ll}
\displaystyle{
\left \langle \sum_{j=1}^d \frac{\dot{\theta}^j_t}{|\dot{\theta}_t|} d M^{ \dot{\theta}^j}_t   \right \rangle } & = \displaystyle{\sum_{j, k=1}^d \frac{\dot{\theta}^j_t}{|\dot{\theta}_t|} \frac{\dot{\theta}^k_t}{|\dot{\theta}_t|}  \left \langle d M^{ \dot{\theta}^j}_t, M^{ \dot{\theta}^k}_t \right \rangle} \\
\\
& =\displaystyle{\sum_{j, k=1}^d \frac{\dot{\theta}^j_t}{|\dot{\theta}_t|} \frac{\dot{\theta}^k_t}{|\dot{\theta}_t|} \left(  \frac{\delta_{j k}}{f^2(r_t)} - \dot{\theta}^j_t \dot{\theta}^k_t \right) } \\
\\
& = \displaystyle{ \frac{1}{f^2(r_t)} - |\dot{\theta}_t|^2},
\end{array}
\]
and 
\[
\begin{array}{ll}
 \displaystyle{
\left \langle d M^{ \dot{\theta}^i}_t, \sum_{j=1}^d \frac{\dot{\theta}^j_t}{|\dot{\theta}_t|} d M^{ \dot{\theta}^j}_t   \right \rangle } & = \displaystyle{\sum_{j=1}^d \frac{\dot{\theta}^j_t}{|\dot{\theta}_t|} \left \langle d M^{ \dot{\theta}^i}_t, M^{ \dot{\theta}^j}_t \right \rangle} \\
\\
& =\displaystyle{\sum_{j=1}^d \frac{\dot{\theta}^j_t}{|\dot{\theta}_t|} \left(  \frac{\delta_{ij}}{f^2(r_t)} - \dot{\theta}^i_t \dot{\theta}^j_t \right) } \\
\\
& = \displaystyle{ \frac{1}{f^2(r_t)} \times \frac{\dot{\theta}^i_t}{|\dot{\theta}_t|}-\frac{\dot{\theta}^i_t}{|\dot{\theta}_t|} \times  |\dot{\theta}_t|^2}.
\end{array}
\]
Injecting the above expressions of the brackets in Equation \eqref{eqn.ap3} yields
\[
\begin{array}{ll}
d \displaystyle{\frac{\dot{\theta}^i_t}{|\dot{\theta}_t|} }=  &  \displaystyle{\frac{1}{f^2(r_t) |\dot{\theta}_t |} \left[\textcolor{green}{ - \frac{(d-1)\sigma^2}{2} f^2(r_t) \dot{\theta}^i_t dt }+ \sigma f^2(r_t) d M^{ \dot{\theta}^i}_t \right]} \\
\\
&  \displaystyle{- \frac{f^2(r_t) \dot{\theta}_t^i }{f^4(r_t) |\dot{\theta}_t |^2} \left[\textcolor{green}{- \frac{(d-1)\sigma^2}{2} f^2(r_t) |\dot{\theta}_t| dt} + \frac{(d-1)\sigma^2}{2|\dot{\theta}_t|} dt +  \sigma f^2(r_t) \sum_{j=1}^d \frac{\dot{\theta}^j_t}{|\dot{\theta}_t|} d M^{ \dot{\theta}^j}_t \right]} \\
\\
&  \displaystyle{+\frac{f^2(r_t) \dot{\theta}_t^i }{f^6(r_t) |\dot{\theta}_t |^3} \left[ \sigma^2 f^4(r_t) \left( \textcolor{blue}{\frac{1}{f^2(r_t)} }- \textcolor{red}{ |\dot{\theta}_t|^2 }\right)\right]  } \\
\\
&  \displaystyle{ - \frac{\sigma^2}{f^4(r_t) |\dot{\theta}_t |^2} \times f^4(r_t) \times \left[ \textcolor{blue}{\frac{1}{f^2(r_t)} \times \frac{\dot{\theta}^i_t}{|\dot{\theta}_t|}}-\textcolor{red}{\frac{\dot{\theta}^i_t}{|\dot{\theta}_t|} \times  |\dot{\theta}_t|^2} \right].}
\end{array}
\]
The expressions with the same color cancel so that
\[
d \frac{\dot{\theta}_t^i}{|\dot{\theta}_t|} = - \sigma^2 \frac{d-1}{2} \frac{\dot{\theta}_t^i}{|\dot{\theta}_t|} \frac{dt}{f^2(r_t)|\dot{\theta}_t|^2 } + \sigma d N^i_t,
\]
with 
\[
d N^i_t := \frac{1}{|\dot{\theta}_t|} \left(  d M^{\dot{\theta}^i}_t  - \frac{\dot{\theta}_t^i}{|\dot{\theta}_t|}  \sum_{j=1}^d \frac{\dot{\theta}_t^j}{|\dot{\theta}_t|} d M^{\dot{\theta}^j}_t\right).
\]
Now, let us compute the bracket $\langle N^i_t, N^j_t\rangle$:
\[
\begin{array}{ll}
d \left \langle 
N^i_t, N^j_t
\right \rangle  & = \displaystyle{\frac{1}{|\dot{\theta}_t|^2} \left \langle d M^{\dot{\theta}^i}_t  - \frac{\dot{\theta}_t^i}{|\dot{\theta}_t|}  \sum_{k=1}^d \frac{\dot{\theta}_t^k}{|\dot{\theta}_t|} d M^{\dot{\theta}^k}_t, d M^{\dot{\theta}^j}_t  - \frac{\dot{\theta}_t^j}{|\dot{\theta}_t|}  \sum_{\ell=1}^d \frac{\dot{\theta}_t^{\ell}}{|\dot{\theta}_t|} d M^{\dot{\theta}^{\ell}}_t\right \rangle}, 
\end{array}
\]
i.e.
\[
\begin{array}{ll}
d \left \langle 
N^i_t, N^j_t
\right \rangle  
& = \displaystyle{\frac{1}{|\dot{\theta}_t|^2} \left[ \left(\frac{\delta_{ij}}{f^2(r_t)} - \dot{\theta}^i_t \dot{\theta}^j_t \right) - \frac{\dot{\theta}_t^j}{|\dot{\theta}_t|} \left \langle d M^{\dot{\theta}^i}_t , \sum_{k=1}^d \frac{\dot{\theta}_t^k}{|\dot{\theta}_t|} d M^{\dot{\theta}^k}_t \right \rangle 
\right]} \\
& \quad  \displaystyle{+ \frac{1}{|\dot{\theta}_t|^2} \left[ -  \frac{\dot{\theta}_t^i}{|\dot{\theta}_t|} \left \langle d M^{\dot{\theta}^j}_t , \sum_{\ell=1}^d \frac{\dot{\theta}_t^{\ell}}{|\dot{\theta}_t|} d M^{\dot{\theta}^{\ell}}_t \right \rangle+  \frac{\dot{\theta}_t^i}{|\dot{\theta}_t|}  \frac{\dot{\theta}_t^j}{|\dot{\theta}_t|} \left \langle \sum_{j=1}^d \frac{\dot{\theta}^j_t}{|\dot{\theta}_t|} d M^{ \dot{\theta}^j}_t   \right \rangle \right]} \\
\\
& = \displaystyle{\frac{1}{|\dot{\theta}_t|^2} \left[ \left(\frac{\delta_{ij}}{f^2(r_t)} - \dot{\theta}^i_t \dot{\theta}^j_t \right) - \frac{\dot{\theta}_t^j}{|\dot{\theta}_t|}  \sum_k \frac{\dot{\theta}_t^k}{|\dot{\theta}_t|} \left(\frac{\delta_{ik}}{f^2(r_t)} - \dot{\theta}^i_t \dot{\theta}^k_t \right) \right]} \\
& \quad + \displaystyle{\frac{1}{|\dot{\theta}_t|^2} \left[- \frac{\dot{\theta}_t^i}{|\dot{\theta}_t|}  \sum_{\ell} \frac{\dot{\theta}_t^{\ell}}{|\dot{\theta}_t|} \left(\frac{\delta_{\ell j}}{f^2(r_t)} - \dot{\theta}^j_t \dot{\theta}^{\ell}_t \right)+ \frac{\dot{\theta}_t^i}{|\dot{\theta}_t|} \frac{\dot{\theta}_t^j}{|\dot{\theta}_t|}  \left(\frac{1}{f^2(r_t)} - |\dot{\theta}_t|^2 \right) \right]}\\
\\
& = \displaystyle{\frac{1}{|\dot{\theta}_t|^2} \left[ \left( \frac{\delta_{ij}}{f^2(r_t)} - \dot{\theta}^i_t \dot{\theta}^j_t\right)  - \frac{\dot{\theta}_t^i}{|\dot{\theta}_t|} \frac{\dot{\theta}_t^j}{|\dot{\theta}_t|} \times \frac{1}{f^2(r_t)} +\dot{\theta}^i_t \dot{\theta}^j_t \right]}\\
\\
& \quad + \displaystyle{\frac{1}{|\dot{\theta}_t|^2} \left[ - \frac{\dot{\theta}_t^i}{|\dot{\theta}_t|} \frac{\dot{\theta}_t^j}{|\dot{\theta}_t|} \times \frac{1}{f^2(r_t)} +\dot{\theta}^i_t \dot{\theta}^j_t +\frac{\dot{\theta}_t^i}{|\dot{\theta}_t|} \frac{\dot{\theta}_t^j}{|\dot{\theta}_t|} \left(\frac{1}{f^2(r_t)} - |\dot{\theta}_t|^2 \right)  \right]}\\
\\
& = \displaystyle{\frac{1}{|\dot{\theta}_t|^2} \left[\frac{\delta_{ij}}{f^2(r_t)}  - \frac{\dot{\theta}_t^i}{|\dot{\theta}_t|} \frac{\dot{\theta}_t^j}{|\dot{\theta}_t|} \times \frac{1}{f^2(r_t)} \right]} \\
\\
& = \displaystyle{\frac{1}{f^2(r_t)|\dot{\theta}_t|^2} \left[\delta_{ij}  - \frac{\dot{\theta}_t^i}{|\dot{\theta}_t|} \frac{\dot{\theta}_t^j}{|\dot{\theta}_t|}\right].}
\end{array}
\]
Hence, the evolution of $\frac{\dot{\theta}_t}{|\dot{\theta}_t|} $ is indeed described by Equations \eqref{eqn.angular} and \eqref{eqn.angularcov}, i.e. the process a time-changed spherical Brownian motion parametrized by the clock $\sigma^2 C_t$
where
\[
C_t = \int_0^t \frac{ds}{f^2(r_s)|\dot{\theta}_s|^2} = \int_0^t \frac{ds}{1-\dot{r}_s^2}.
\]
Moreover, for $1 \leq i \leq d$, we have 
\[
\begin{array}{ll}
\left \langle d N^i_t, d M^{\dot{r}}_t \right \rangle  & = \displaystyle{\frac{1}{|\dot{\theta}_t|} \left[  \left \langle d M^{\dot{r}}_t , d M^{\dot{\theta^i}}_t \right \rangle - \frac{\dot{\theta}_t^i}{|\dot{\theta}_t|} \sum_{k=1}^d   \frac{\dot{\theta}_t^k}{|\dot{\theta}_t|} \left \langle d M^{\dot{r}}_t , d M^{\dot{\theta^k}}_t \right \rangle \right]} \\
\\
& = \displaystyle{\frac{1}{|\dot{\theta}_t|} \left[ -\dot{r}_t \dot{\theta}^i_t - \frac{\dot{\theta}_t^i}{|\dot{\theta}_t|} \sum_{k=1}^d   \frac{\dot{\theta}_t^k}{|\dot{\theta}_t|} ( -\dot{r}_t \dot{\theta}^k_t)\right]} \\
\\
& = \displaystyle{\frac{1}{|\dot{\theta}_t|} \left[ -\dot{r}_t \dot{\theta}^i_t +\dot{r}_t \dot{\theta}^i_t \right]=0.} 
\end{array}
\]
Thus, the above spherical Brownian motion describing the evolution of $\frac{\dot{\theta}_t}{|\dot{\theta}_t|} $ is independant of the radial process, in particular, it is independant of the clock $C_t$.

}\fi

\bibliographystyle{alpha}

\end{document}